\documentclass{amsart}
\usepackage{amssymb,amsfonts}
\usepackage{latexsym}
\usepackage{amsmath,mathrsfs}
\usepackage[all,arc]{xy}
\usepackage{enumerate}
\usepackage[english]{babel}
\usepackage[bookmarks=true]{hyperref}
\usepackage{tikz}
\newtheorem{theo}{Theorem}[section]
\newtheorem{coro}[theo]{Corollary}
\newtheorem{prop}[theo]{Proposition}
\newtheorem{lemm}[theo]{Lemma}

\newtheorem*{obse}{Observation}

\theoremstyle{definition}
\newtheorem{defi}[theo]{Definition}

\theoremstyle{remark}
\newtheorem{rema}[theo]{Remark}

\date{\today}
\newcommand{\bb}[1]{\mathbb{#1}}
\newcommand{\al}[1]{\mathcal{#1}}
\newcommand{\sr}[1]{\mathscr{#1}}
\newcommand{\ak}[1]{\mathfrak{#1}}
\newcommand{\gen}[1]{\left\langle #1\right\rangle}

\newcommand{\ra}{\rightarrow}
\newcommand{\lra}{\longrightarrow}

\begin{document}
\title{Hitchin Systems for Invariant and Anti-invariant vector bundles}
\author{ZELACI Hacen}
\address{Laboratoire de Math\'ematiques J.A. Dieudonn\'e.}
\curraddr{}
\email{z.hacen@gmail.com}
\date{\today}


\begin{abstract}
Given a smooth projective complex curve $X$ with an involution $\sigma$, we study the Hitchin systems for the locus of anti-invariant (resp. invariant) stable vector bundles over $X$  under $\sigma$. Using these integrable systems and the theory of the nilpotent cone,  we study the irreducibility of these loci. The anti-invariant locus can be thought of as a generalisation  of Prym varieties to  higher rank. 
\end{abstract}
\maketitle
\tableofcontents

\section{Introduction}
The Hitchin systems are integrable algebraic systems defined on the cotangent space of the moduli space of stable $G-$bundles on a Riemann surface.  They lie at the crossroads between algebraic geometry, the theory of Lie algebras and the theory integrable systems. Consider a  smooth projective irreducible curve $X$ of genus $g_X\geqslant2$. Hitchin in \cite{NH} has defined and studied some integrable systems related to the moduli space of stable $G-$bundles over $X$, where $G$ is a classical algebraic group ($\text{GL}_r$, $\text{Sp}_{2m}$ and $\text{SO}_r$). Let $\al M_X(G)$ be this moduli space, the tangent space to $\al M_X(G)$ at a point $[E]$ can be identified with $$H^1(X,\text{Ad}(E))\cong H^0(X,\text{Ad}(E)\otimes K_X)^*,$$ where $\text{Ad}(E)$ is the adjoint bundle associated to $E$, which is a bundle of Lie algebras isomorphic to $\ak g=\text{Lie}(G)$. By Serre duality, the fiber of the cotangent bundle is $H^0(X,\text{Ad}(E)\otimes K_X)$. By considering a basis of the invariant polynomials on $\ak g$, one gets a map $$T^*_E\al M_X(G)=H^0(X,\text{Ad}(E)\otimes K_X)\longrightarrow \bigoplus_{i=1}^k H^0(X,K_X^{d_i}),$$ where the $(d_i)_i$ are the degrees of these invariant polynomials. Hitchin has shown that these two spaces have the same dimension.\\ In the case $G=\text{GL}_r$, a basis of the invariant polynomials is given by the coefficients of the characteristic polynomial. If $E$ is a stable vector bundle, then this gives rise to a map $$\sr H_E:H^0(X,\text{End}(E)\otimes K_X)\lra \bigoplus_{i=1}^r H^0(X,K_X^i)=:W,$$ 
which associates to each Higgs field $\phi$, the coefficients of its characteristic  polynomial. The associated map $$\sr H:T^*\al M_X(\text{GL}_r)\lra W$$ is called the \emph{Hitchin morphism}. By choosing a basis of $W$, $\sr H$ is represented by  $d=r^2(g_X-1)+1$ functions $f_1,\dots, f_d$. Hitchin has proved that this system is  \emph{algebraically completely integrable}, i.e. its generic fiber is an open set in an abelian variety of dimension $d$, and the vector fields $\al X_{f_1},\dots,\al X_{f_d}$ associated to $f_1,\cdots,f_d$ (defined using the canonical $2-$form on $T^*\al M_X(\text{GL}_r)$) are linear.

Moreover, let $\al U_X(r,0)$ be the moduli space of stable vector bundles of rank $r$ and degree $0$ on $X$. Consider the map $$\Pi:T^*\al U_X(r,0)\ra \al U_X(r,0)\times W$$ whose first factor is the canonical projection and the second factor is $\sr H$. Then, in \cite{BNR}, it is proved that $\Pi$  is dominant. \\

   Suppose now that we have an involution $\sigma$  on $X$. It induces by pullback an involution on the moduli space of stable vector bundles $\al U_X(r,d)$. Let $E$ be a stable vector bundle, we say that $E$ is \emph{anti-invariant} if there exists an isomorphism $$\psi:\sigma^*E\stackrel{\sim}{\lra} E^*.$$ We say that the anti-invariant vector bundle $E$ is \emph{$\sigma-$symmetric} (resp. \emph{$\sigma-$alternating})  if $\sigma^*\psi=\,^t\psi$ (resp. $\sigma^*\psi=-\,^t\psi$). We denote by $\al U_X^{\sigma,+}(r)$ and $\al U_X^{\sigma,-}(r)$ the loci of $\sigma-$symmetric and $\sigma-$alternating anti-invariant vector bundles respectively. We define the invariant locus to be  $$\al U_X^\sigma(r,d)=\{[E]\in\al U_X(r,d)\,|\,\sigma^*E\cong E\}.$$ These varieties correspond to moduli spaces of the form $\al M_Y(\al G)$, i.e. moduli spaces of $\al G-$torsors over $Y:=X/\sigma$ for some particular type of group schemes $\al G$ (they are called parahoric Bruhat-Tits group schemes, see \cite{PR}, \cite{He} and \cite{BS}). 
   
   The main topic of this paper is the study of the Hitchin systems for the anti-invariant and the invariant loci. We use these systems to identify the connected components of $\al U_X^{\sigma,+}(r)$ and $\al U_X^{\sigma,-}(r)$. The irreducibility  of the invariant locus (of a fixed type) is already know in more general setting (see \cite{BS}).\\
In this paper,  by \emph{Prym variety} of a cover of curves $\pi:\bar X\ra \bar Y$ we mean the kernel of the norm map $\text{Nm}:J_{\bar X}\lra J_{\bar Y}$ attached to $\pi$, which is in general nonconnected (hence it is not an abelian variety).\\
Our main result, in the case of the anti-invariant vector bundles, can be formulated as follows 
\begin{theo}\label{main1.1}
Assume that $\pi:X\ra Y$ is ramified. We have the following results:
\begin{enumerate}
 \item There exists a linear subspace \mbox{$W^{\sigma,+}\subset W$}, such that $\Pi$ induces a dominant map $$\Pi:T^*\al U_X^{\sigma,+}(r)\longrightarrow \al U_X^{\sigma,+}(r)\times W^{\sigma,+}.$$
 Moreover, for general $s\in W^{\sigma,+}$, let $q:\tilde{X}_s \ra X$ be the associated spectral curve over $X$, then  $\tilde{X}_s$ is smooth and the involution $\sigma$ lifts to an involution $\tilde{\sigma}$ on $\tilde{X}_s $ such that, if $\al P^+$ is some translate of the Prym variety of  $\tilde{X}_s \ra \tilde{X}_s /\tilde{\sigma}$, then the pushforward  rational map $$q_*:\al P^+\ra \al U_X^{\sigma,+}(r)$$ is dominant. In particular  $\al U_X^{\sigma,+}(r)$ is irreducible.
 \item Suppose that $r$ is even, then there exists an affine subvariety $W^{\sigma,-}\subset W^{\sigma,+}$, such that $\Pi$ induces a dominant map $$\Pi:T^*\al U_X^{\sigma,-}(r)\longrightarrow \al U_X^{\sigma,-}(r)\times W^{\sigma,-}$$ 
such that for general $s\in W^{\sigma,-}$, the associated spectral curve is singular, and if $\hat{X}_s$ is its normalisation, then we have a dominant map $$\hat{q}_*:\hat{\al P}\ra \al U_X^{\sigma,-}(r),$$
where  $\hat{\al P}$ is some translate of the Prym variety of $\hat X\ra\hat X/\hat{\sigma}$, where $\hat\sigma$ is the lifting of the involution $\tilde{\sigma}:\tilde{X}_s \ra \tilde{X}_s $ to $\hat{X}_s$. In particular, as $\hat{\sigma}$ has no fixed point, we deduce that $\al U_X^{\sigma,-}(r)$ has two connected components. 
 \end{enumerate}
\end{theo} 
The two connected components of $\al U_X^{\sigma,-}(r)$ are distinguished by a cohomological criterion (similar to that given by  the Stiefel-Whitney class, see subsection \ref{ramifiedalterne} for more details).

The \'etale case is little special. Assume that $\pi:X\ra Y$ is \'etale, denote by $\Delta=\text{det}(\pi_*\al O_X)^{-1}$ the $2-$torsion line bundle associated to $\pi$, then we have  the following results
\begin{theo}\label{main1.2}
 Choosing a line bundle on $X$ of norm $\Delta$ induces, by tensor product, an isomorphism $$\al U_X^{\sigma,+}(r)\stackrel{\sim}{\lra}\al U_X^{\sigma,-}(r),$$ and we have $W^{\sigma,+}=W^{\sigma,-}$. Furthermore  $\Pi$ induces dominant maps $$T^*\al U_X^{\sigma,+}(r)\lra \al U_X^{\sigma,+}(r)\times W^{\sigma,+},$$ 
$$T^*\al U_X^{\sigma,-}(r)\lra \al U_X^{\sigma,-}(r)\times W^{\sigma,-},$$ 
And for general $s\in W^{\sigma,+}$, the pushforword map $q_*$ induces dominant rational maps $$\al P^+\dashrightarrow \al U_X^{\sigma,+}(r),\;\;\al P^-\dashrightarrow \al U_X^{\sigma,-}(r),$$
where $\al P^+,\;\al P^-$ are different translates of the Prym variety of $\tilde{X}_s\ra \tilde{X}_s/\tilde{\sigma}$. \\
In particular both $\al U_X^{\sigma,+}(r)$ and  $\al U_X^{\sigma,-}(r)$ have two connected components.\\
\end{theo}

\begin{rema} 
\begin{enumerate}
\item If $r$ is odd and $\pi:X\ra Y$ is ramified, then $\al U_X^{\sigma,-}(r)=\emptyset$. In all other cases, the spaces $\al U_X^{\sigma,+}(r)$ and $\al U_X^{\sigma,-}(r)$ are non-empty (see \ref{appA} for the construction of stable anti-invariant vector bundles). 
\item If $r=1$, then $\al U_X^{\sigma,+}(1)=\text{Prym}(X/Y)$, and if $X\ra Y$ is \'etale, then $\al U_X^{\sigma,-}(1)=\text{Nm}^{-1}(\Delta)$.
\end{enumerate}
\end{rema}

The case of stable $\sigma-$symmetric vector bundles with trivial determinant, denoted $\al{SU}_X^{\sigma,+}(r)$, can then be deduced 
\begin{coro}
For general $s\in W^{\sigma,+}$, the direct image rational map $$\al P^+\cap\al Q^+\dashrightarrow \al{SU}_X^{\sigma,+}(r) $$ is dominant, where $\al Q^+$ is some translate of the Prym variety of the spectral cover $\tilde{X}_s \ra X$. In particular we deduce that $\al{SU}_X^{\sigma,+}(r)$ is irreducible.
\end{coro}

We also prove similar results for the locus  $\al U_X^{\sigma,\tau}(r,d)$ of  $\sigma-$invariant vector bundles of a fixed type $\tau$.  
More precisely, for any type $\tau$, we characterise a linear subspace $W^{\sigma,\tau}\subset W$ such that the Hitchin morphism $\sr H$ induces a map $$T^*\al U_X^{\sigma,\tau}(r,0)\lra  W^{\sigma,\tau}$$ and $\text{dim}(W^{\sigma,\tau})=\text{dim}(\al U_X^{\sigma,\tau}(r,0))$. We prove again that this system gives a dominant rational map from a translate of the Jacobian of the normalisation $\hat{X}_s$ of the spectral curve attached to general $s\in W^{\sigma,\tau}$ $$\text{Pic}^{\bullet}(\hat{X}_s)^{\tilde{\sigma},\hat{\tau}}\dashrightarrow \al U_X^{\sigma,\tau}(r,0),$$
for some type $\hat{\tau}$ of $\hat{\sigma}-$invariant line bundles over $\hat{X}_s$. \\

The locus of $\sigma-$invariants vector bundles of fixed type $\tau$ corresponds to parabolic vector bundles of parabolic structure attached to $\tau$ over the quotient curve $Y$. The Hitchin systems over the moduli space of parabolic vector bundles have been studied by Logares and Martens \cite{LM} in more general situation. \\
More recently, Baraglia, Kamgarpour and Varma in \cite{BKV}, have studied the complete integrability of the Hitchin systems over the moduli of parahoric $\al G-$bundles, for a non-twisted parahoric group scheme $\al G$. This can be thought of as a generalisation of the parabolic bundles case. However the anti-invariant vector bundle case corresponds to parahoric $\al G-$bundles over $Y$ for a \emph{twisted} parahoric group scheme $\al G$.

Some other authors have also considered some closely related problems. We mention the thesis of Schaposnik (\cite{Sch}), where she studied the Hitchin system for invariant Higgs bundles under involutions induced by real forms of groups. Also Andersen and Grove have studied in \cite{AJ} invariant rank two vector bundles under the action of an automorphism of the base curve.\\

\textbf{Plan of the paper.} We start in section $2$ by recalling the main results of the theory of spectral curves and the Hitchin systems. In section $3$ we will give the basic facts about the invariant and anti-invariant vector bundles and study their loci. This is studied in more detail in my thesis. In section $4$ we study the Hitchin system for the anti-invariant  case and give the proof of Theorems \ref{main1.1} and \ref{main1.2}. In the last section we concentrate on the Hitchin systems for the invariant locus.\\

\textbf{Acknowledgement.} I am very grateful to my advisor Prof. Christian PAULY, who has led me during the preparation of this work, which, for sure, would not have been possible without his advice and encouragement.\\

\section{Preliminaries}
In this section we recall  the general theory of spectral curves. Our main reference is \cite{BNR}. The ground field is always assumed to be   $\bb C$. 

Let $L$ be any line bundle  over a smooth projective curve $X$. Consider the ruled surface over $X$ $$\bar q: \bb S=\bb P(\al O_X\oplus L^{-1})\ra X,$$ where for a vector bundle $\sr E$ we denote $\text{Sym}^\bullet(\sr E)$ the symmetric algebra and $$\bb P(\sr E)=\text{Proj}(\text{Sym}^\bullet(\sr E)).$$ Hence a point in  $\bb S$ lying over $x\in X$ corresponds to a hyperplane in the fiber  $(\al O_X\oplus L^{-1})_x$. It follows that the total space of $L$ denoted $|L|$ is contained in $\bb S$.

Let $\al O(1)$ be the relatively ample line bundle over $\bb S$. It is well known that $\bar q_*\al O(1)\cong \al O_X\oplus L^{-1}$.  Hence $\al O(1)$ has a canonical section, denoted by $y$, corresponding to the direct summand $\al O_X$. Also by the projection formula $\bar q_*(\bar q^*L\otimes\al O(1))$ is isomorphic to $L\oplus \al O_X$, so it has a canonical section which we denote by $x$. 

Let $$s=(s_1,\cdots,s_r)\in\bigoplus_{i=1}^r H^0(X,L^i)=:W_L$$ be an $r-$tuple  of global sections of $L^i$ and consider the global section
\begin{equation} \label{sectiontildeX} x^r+(\bar q^*s_1)yx^{r-1}+\cdots+ (\bar q^*s_r)y^r \in H^0(\bb S,\bar{q}^*K_X^r\otimes \al O(r)).\end{equation} 
We denote by $\tilde{X}_s$ its zero scheme which is a curve. We say that $\tilde{X}_s $ is the \emph{spectral curve} associated to $s\in W_L$. Denote $q:\tilde{X}_s \ra X$ the restriction of $\bar q$ to $\tilde{X}_s $. It is  clear that $\tilde{X}_s $ is finite cover of degree $r$ of $X$ and its fiber over $p\in X$ is given by the homogeneous equation in $\bb P^1$  $$x^r+s_1(p)x^{r-1}y\cdots+s_r(p)y^r=0.$$
\begin{lemm} \label{smoothlemma} The set of elements $s\in W_L$ corresponding to smooth spectral curves $\tilde{X}_s$ is open.  In particular it is dense whenever it is not empty. 
\end{lemm} 
\begin{proof} Assume that $\tilde{X}_s $ is integral (i.e. reduced and irreducible, which is true for general $s\in W$, see \cite{BNR}, Remark 3.1)  and let $$P(x,t)=x^r+s_1(t)x^{r-1}+\cdots+s_r(t)=0$$ be the equation of $\tilde{X}_s $ locally over a point $p\in X$, where $t$ is a local parameter near $p$. Then, by the Jacobian criterion of smoothness,   $\tilde{X}_s $ is singular at a point $\lambda\in \tilde{X}_s $ over $p$ if and only if $$\dfrac{\partial P}{\partial x}(\lambda,0)=\dfrac{\partial P}{\partial t}(\lambda,0)=0,$$
i.e. $$ r\lambda^{r-1}+(r-1)s_1(0)\lambda^{r-2}+\cdots+s_{r-1}(0)=0, $$
     $$ s_1'(0)\lambda^{r-1}+s_2'(0)\lambda^{r-2}+\cdots+s_{r}'(0)=0. $$
Clearly these two equations give a closed condition on $s=(s_1,\cdots,s_r)\in W_L$. Hence the set of $s\in W_L$ corresponding to  smooth curves $\tilde{X}_s $ is open.
\end{proof} 
\begin{rema} 
We remark that the criterion of smoothness given in \cite{BNR}, Remark 3.5, is not correct. In fact this criterion assumes that the singular point is located at $\lambda=0$.
\end{rema}
\begin{rema}
An alternative way to construction $\tilde{X}_s $ is as follows: consider the symmetric $\al O_X-$algebra $\text{Sym}^\bullet(L^{-1})$. Define the ideal    
$$\ak I=\gen{\bigoplus_i s_i(L^{-r})}\subset \text{Sym}^\bullet(L^{-1}),$$ where $s_i\in H^0(X,L^i)$ is seen here as an embedding $s_i:L^{-r}\ra L^{-r+i}$. Then $\tilde{X}_s $ can be defined as $\text{Spec}\left(\text{Sym}^\bullet(L^{-1})/\ak I\right)$.                                                                                                                                                                                                                                                                                                                                                                                                                                                                                                                                                                                                                                         
\end{rema}
Suppose that $\tilde{X}_s $ is smooth and let $\tilde{S}=Ram(\tilde{X}_s /X)\subset\tilde{X}_s $ be the ramification divisor of $q:\tilde{X}_s \ra X$.\\
Recall that $$q_*\al O_{\tilde{X}_s }\cong\al O_X\oplus L^{-1}\oplus\cdots\oplus L^{-(r-1)},$$
hence, by duality of finite flat morphisms (see  \cite{HA}, Ex III.6.10) $$q_*\left(\al O_{\tilde{X}_s }(\tilde{S})\right)\cong\left(q_*\al O_{\tilde{X}_s }\right)^*\cong \al O_X\oplus L\oplus\cdots\oplus L^{r-1}.$$  In particular, using the fact that for any line bundle $M$ over $\tilde{X}_s$ $$\text{det}(q_*M)=\text{det}(q_*\al O_{\tilde{X}_s })\otimes \text{Nm}_{\tilde{X}_s/X}(M),$$
where $\text{Nm}_{\tilde{X}_s/X}:\text{Pic}(\tilde{X}_s )\ra \text{Pic}(X)$ is the norm map, we deduce $$\text{deg}(\tilde{S})=r(r-1)\text{deg}(L).$$
Furthermore, by Hurwitz formula, we have $K_{\tilde{X}_s }=q^*K_X(\tilde{S})$. Thus, by the projection formula we get $$q_*K_{\tilde{X}_s }\cong K_X\oplus K_XL\oplus\cdots\oplus K_XL^{r-1}.$$
It follows that the genus $g_{\tilde{X}_s}$ of $\tilde{X}_s $ is $$g_{\tilde{X}_s }=\text{deg}(L)\dfrac{r(r-1)}{2}+r(g_X-1)+1.$$\\
Recall that for a stable vector bundle $E$, the Hitchin map  $$\sr H_E:H^0(X,E\otimes E^*\otimes L)\ra W_L$$ is defined by $$s\lra \sr H_E(s)=\left((-1)^i \text{Tr}(\bigwedge^i s)\right)_i,$$ where $\text{Tr}$ is the trace map.\\
We recall a very important result from \cite{BNR}
\begin{prop} \label{bnrprop}
Let $\tilde{X}_s $ be an \emph{integral} (resp. smooth) spectral curve over $X$ associated to $s\in W_L$. Then there is a one-to-one correspondence between torsion-free $\al O_{\tilde{X}_s}-$modules of rank $1$ (resp. $\text{Pic}(\tilde{X}_s )$) and the isomorphism classes of pairs $(E,\phi)$ where $E$ is a rank $r$ vector bundle and $\phi:E\ra E\otimes L$ is a morphism such that $\sr H_E(\phi)=s$
\end{prop}

Maybe the most important case of spectral curves is when $L=K_X$. We denote simply by $W$ the space $W_{K_X}$.  In this case, the genus $g_{\tilde{X}_s}$ of $\tilde{X}_s $ is $g_{\tilde{X}_s}=r^2(g_X-1)+1$, which coincides with the dimension of the moduli space $\al U_X(r,0)$ of stable vector bundles of rank $r$ and degree $0$ over $X$. In \cite{BNR} it is proved that the map  $$\Pi:T^*\al U_X(r,0)\ra \al U_X(r,0)\times W$$ is dominant. Moreover, the fiber $\sr H^{-1}(s)$ of a general point $s\in W$ is isomorphic to an open subset of $\text{Pic}^{m}(\tilde{X}_s )$, where $m=r(r-1)(g_X-1)$. We claim that this is still true for the classical algebraic groups $\text{Sp}_{2m}$ et $\text{SO}_r$. Consider the moduli spaces $\al M_X(\text{Sp}_{2m})$ and $\al M_X(\text{SO}_{r})$ of $\text{Sp}_{2m}-$bundles and $\text{SO}_r-$bundles respectively which are \emph{stable as vector bundles}. Define  $$W_{Sp_{2m}}=\bigoplus_{i=1}^mH^0(X,K_X^{2i}),$$ 
and $$W_{SO_r}=\begin{cases} \bigoplus_{i=1}^{r/2-1}H^0(X,K_X^{2i})\oplus H^0(X,K_X^{r/2}) & r\equiv 0\mod2 \\ \bigoplus_{i=1}^{(r-1)/2}H^0(X,K_X^{2i})& r\equiv 1\mod 2 \end{cases}.$$ 
For general $s\in W_{Sp_{2m}}$ the curve $\tilde{X}_s$ is smooth, and for general $s\in W_{SO_r}$ the associated $\tilde{X}_s$ is nodal curve. In this case we denote $\hat{X}_s$ its normalisation. In both cases, the involution of the ruled surface $\bb S$ that sends $x$ to $-x$  induces an involution on $\tilde{X}_s$, we denote it by $\iota$. Remark that in the singular case, $\iota$ lifts to an involution on $\hat{X}_s$ without fixed points.\\
Recall that Hitchin (\cite{NH}) has proved that the map $\Pi$ induces maps  
\begin{gather} \label{pidominant}
\begin{aligned}
T^*\al M_X(\text{Sp}_{2m})\lra \al M_X(\text{Sp}_{2m})\times W_{Sp_{2m}}, \\
T^*\al M_X(\text{SO}_r)\lra \al M_X(\text{SO}_r)\times W_{SO_r}.
\end{aligned}
\end{gather}
Moreover, the generic fiber in the case of symplectic bundles is isomorphic to an open set of a translate of the Prym variety  of  $\tilde{X}_s \ra \tilde{X}_s /\iota$. In the case of orthogonal bundles, the generic fiber is an open dense of the Prym variety of $\hat{X}_s\ra \hat{X}_s/\iota$. We refer to \cite{NH} for more details.  
\begin{prop} 
The restrictions of $\Pi$ given in (\ref{pidominant}) are dominant. Moreover, for general $s\in W_{Sp_{2m}}$ (resp. $s\in W_{SO_r}$), if $\al P$ is a translation of the Prym variety of $\tilde{X}_s\ra \tilde{X}_s/\iota$ (resp. $\hat{X}_s\ra \hat{X}_s/\iota$), then the pushforward map $$\al P\dashrightarrow \al M_X(\text{Sp}_{2m})\;\;(\text{resp. } \al M_X(\text{SO}_r))$$ is dominant.  
\end{prop}
\begin{proof}
Laumon has proved in \cite{LG} that the nilpotent cone $$\Lambda_G\subset T^*\al M_X(G)$$ is Lagrangian, for any reductive algebraic group $G$. In particular, for $G=\text{Sp}_{2m}$ (resp. $G=\text{SO}_r$), we deduce that the locus of $G-$bundles $E$ such that $$\sr H_E:H^0(X,\text{Ad}(E)\otimes K_X)\ra W_{Sp_{2m}}\;\;( \,\text{resp. }  W_{SO_r})$$ is dominant, forms an open dense  subset of $\al M_X(G)$. Indeed, we have $$\text{dim}(\Lambda_G)=\text{dim}(\al M_X(G)),$$ and the restriction of the canonical projection $T^*\al M_X(G)\ra \al M_X(G)$ to $\Lambda_G$ is surjective (because $(E,0)\in \Lambda_G$ for any $G-$bundle $E$). Hence by the dimension theorem, it follows that there exists an open dense subset of $\al M_X(G)$ over which $\Lambda_G$ is reduced to the zero section of $T^*\al M_X(G)$. This open subset is by definition the set of very stable bundles $E$, for which, the map $\sr H_E$ is dominant.\\
It follows that the restrictions of $\Pi$ given in (\ref{pidominant}) are dominant maps. Hence for general $s\in W_{Sp_{2m}}$ (resp. $s\in W_{SO_r}$), we get a dominant maps
$$\sr H^{-1}(s)\lra \al M_X(\text{Sp}_{2m})\;\;(\text{resp. } \al M_X(\text{SO}_r)).$$
Furthermore, if $S$ is the ramification of $\tilde{X}_s/\iota\ra X$ (resp. $\hat{X}_s/ \iota\ra X$),   $\al P=Nm^{-1}(\al O(S))$, where $Nm$ is the norm map attached to the cover $\tilde{X}_s \ra \tilde{X}_s /\iota$ (resp. $\hat{X}_s \ra \hat{X}_s /\iota$), then, by \cite{NH}, $\sr H^{-1}(s)$ is an open dense of $\al P$. Thus the pushforward map $$\al P\dashrightarrow \al M_X(\text{Sp}_{2m})\;\;\;(\text{resp. } \;\al M_X(\text{SO}_r))$$ is dominant rational map. \\
Remark that in the symplectic case, the involution $\iota$ has some fixed points, this implies that $\al P$ is irreducible. While in the orthogonal case, $\iota$ is \'etale, hence $\al P$ has two connected components, each one of them dominates a connected component of $\al M_X(\text{SO}_r)$. In particular we deduce a cohomological criterion identifying the two connected components of $\al M_X(\text{SO}_r)$, more explicitly, take an even theta characteristic $\kappa$ of $X$, then the two components are distinguished by the parity of $h^0(X,E\otimes\kappa).$ This is the same as the criterion given by the Stiefel-Whitney class (see for example \cite{BA}).
\end{proof}

\section{Invariant and anti-invariant vector bundles}
 Let $X$ be a smooth projective irreducible curve with an involution $\sigma:X\ra X$. Let $Y:=X/\sigma$ be the quotient, which is smooth, and denote by $\pi:X\ra Y$ the double cover map. Denote by $R\subset X$ the ramification locus of $\pi$, and let $\text{deg}(R)=2n$. If $g_X$ and $g_Y$ are the genus of $X$ and $Y$, then by Hurwitz formula we have $$g_X=2g_Y-1+n.$$
We denote by $\al U_X(r,d)$ the moduli space of \emph{stable} vector bundles of rank $r$ and degree $d$ over $X$.

Let $Z$ be a smooth variety over $\bb C$ with an involution $\tau$. We define the fixed locus of  $\tau$, denoted $Z^\tau$, to be the intersection of the diagonal $Z\subset Z\times Z$ with the graph $\Gamma_\tau\subset Z\times Z$ of the involution $\tau$.
It is clearly a closed subvariety of $Z$. Moreover, we have
\begin{lemm} \label{smoothnesslemma}
The fixed locus $Z^\tau$ is smooth.
\end{lemm}
\begin{proof} 
The action can be linearised locally around any point $z\in Z^{\tau}$. In fact, this is true in more general context (see Edixhoven \cite{EB}).  
\end{proof}

\subsection{Invariant vector bundles} \label{invariant}
\subsubsection{}A vector  bundle $E$ on $X$ is called \emph{$\sigma-$invariant} (or simply \emph{invariant}) if there exists an isomorphism $$\varphi:\sigma^*E\stackrel{\sim}{\lra} E.$$  
A \emph{linearisation} of the $\sigma-$action on $E$ is an isomorphism  $$\varphi:\sigma^*E\stackrel{\sim}{\lra} E$$ 
such that $\varphi\circ\sigma^*\varphi=id_E$. Hence a linearisation corresponds to a lifting of the involution $\sigma$ to an involution $\tilde{\sigma} :E\ra E$, such that the following diagram 
$$\xymatrix{E\ar[dr]^{\varphi^{-1}} \ar@/^1pc/[rrd]^{\tilde{\sigma}} \ar@/_1pc/[rdd] & &  \\ & \sigma^*E\ar[r]\ar[d] & E\ar[d] \\ & X\ar[r]^\sigma &X}$$ commutes. Using the linearisation $\varphi$ we obtain a linear involution on the space of global sections of $E$, given by $$s\lra \varphi(\sigma^*s).$$ 
We denote its proper subspaces by $H^0(X,E)_\pm$.
\begin{rema} If $E$ is stable $\sigma-$invariant, there are only $2$ linearisations, $\varphi$ and $-\varphi$.  
\end{rema}
Suppose that $E$ is a $\sigma-$invariant \emph{stable} vector bundle and $\varphi:\sigma^*E\ra E$. We define the \emph{type} of $E$ to be $$\tau=\left(\varphi_p\right)_{p\in R} \mod \pm I_r,$$ with $\varphi_p\in \text{End}(E_p)$. We denote usually by $k_p$ the multiplicity of the $-1$ eigenvalue of $\varphi_p$. Note that the vectors $(k_p)_p$  and $(r-k_p)_p$ represent the same type (due to multiplication by $-1$). Moreover, by looking at the determinant of $E$, which is $\sigma-$invariant, we obtain the following relation between the type and the degree $d$ of $E$  $$\sum_{p\in R}k_p\equiv d\mod2.$$
Indeed, define $F$ to be the kernel of $$0\ra F\ra E\ra \bigoplus_{p\in R}(E_p)_-\ra 0.$$ By Kempf's Lemma (see Lemma \ref{kempf} below), it follows that $F$ descends to $Y$, hence $$d-\sum_{p\in R}k_p=\text{deg}(F)\equiv 0\mod2.$$
Denote by $\al U_X^{\sigma,\tau}(r,d)\subset \al U_X(r,d)$ the locus of classes $[E]\in \al U_X(r,0)$ such that $E$ is  $\sigma-$invariant stable vector bundle of type $\tau$. Note that $\al U_X^{\sigma,\tau}(r,d)$ is smooth by Lemma \ref{smoothnesslemma}.\\
The following lemma is very useful (see \cite{ND}):  
\begin{lemm}{(Kempf Lemma)}\label{kempf}
Let $(E,\varphi)$ be a $\sigma-$linearised vector bundle on $X$ (that's a $\sigma-$invariant vector bundle with a linearisation). Then $E$ descends to $Y$ (i.e $E\cong \pi^* F$ for some vector bundle $F$ on $Y$, and $\varphi$ is the canonical associated linearisation) if and only if $\varphi_p=id$, for any $p\in R$.
\end{lemm}
A consequence of this lemma is the following 
\begin{lemm} \label{descentcanonical} The canonical line bundle $K_X$ of $X$   descends to $Y$.
\end{lemm} 
\begin{proof}
By differentiating the involution $\sigma:X\ra X$ we get a linear isomorphism $$d\sigma:K_X^{-1}\ra \sigma^*K_X^{-1}.$$ As $\sigma\circ\sigma=id$, it follows $$d\sigma\circ\sigma^*(d\sigma)=id.$$
Hence $d\sigma$ is a linearisation of $K_X^{-1}$. Moreover, if $t$ is a local parameter near a ramification point $p\in R$, then $\sigma(t)=-t$, hence $d\sigma=-1$. So by Lemma \ref{kempf} we deduce that $K_X$ descends to $Y$.
\end{proof}
In particular, as by Hurwitz formula $\al O_X(R)=K_X\otimes \pi^* (K_Y^{-1})$, we deduce that $\al O_X(R)$ descends to $Y$. Furthermore, if we denote by $\Delta$ the line bundle on $Y$ such that $$\pi_*\al O_X=\al O_Y\oplus \Delta^{-1},$$ then $\al O_X(R)=\pi^*\Delta$, hence $K_X=\pi^*(K_Y\otimes\Delta)$.

\begin{rema} \label{propersubspaces}
Suppose that $\pi$ is ramified. Let $L$ be a line bundle on $Y$, then $\pi^*L$ has a canonical linearisation. We call it the  \emph{positive} linearisation, (because it equals $+id$ over each $p\in R$). Its opposite is called the \emph{negative} linearisation.  Moreover, with respect to the positive linearisation, we have $$H^0(X,\pi^*L)_+\cong H^0(Y,L),\;\;H^0(X,\pi^*L)_-\cong H^0(Y,L\otimes\Delta^{-1}).$$
If $\pi:X\ra Y$ is \'etale, then  $K_X=\pi^* K_Y=\pi^*(K_Y\otimes\Delta)$. We define the positive linearisation on $K_X$ to be the linearisation attached to $K_Y\otimes \Delta$.\\
\end{rema}

\subsubsection{Infinitesimal study}\label{infinvariant}
The tangent space to the moduli space $\al U_X(r,d)$ at a smooth point $E$ is given by $$\text{T}_E \al U_X(r,d)\cong H^1( X, \text{End}(E)),$$
where $\text{End}(E)\cong E\otimes E^*$ stands for the sheaf of endomorphisms of $E$. \\

Recall that a deformation of $E$ aver $Spec(\bb C[\varepsilon])$ ($\varepsilon^2=0$) is defined to be a locally free coherent sheaf $\sr E$ on $ X_\varepsilon= X\times Spec(\bb C[\varepsilon])$ together with a homomorphism $\sr E\ra E$ of $\al O_{ X_\varepsilon}-$module, such that  the induced map $\sr E\otimes \al O_{ X}\ra E$ is an isomorphism. Canonically, the set of deformation of $E$ over $Spec(\bb C[\varepsilon])$ is isomorphic to $H^1( X,\text{End}(E))$. As by definition, a deformation is locally free, so it is flat, thus taking the tensor product of the exact sequence  $$0\ra \al O_{ X}\stackrel{\varepsilon}{\ra }\al O_{ X_\varepsilon} \ra \al O_{ X}\ra 0$$ by $\sr E$ we get   $$0\ra E \stackrel{\varepsilon}{ \ra}\sr E\ra E\ra 0.$$

Assume now that $E$ is stable $\sigma-$invariant vector bundle of rank $r$ and degree $d$, let $\tau$ be its type. We want to identify the tangent space to $\al U_X^{\sigma,\tau}(r,d)$ at $E$.\\ 
 The linearisation $\varphi:\sigma^*E\ra E$  induces a linear involution $f$ on the tangent space $T_E\al U_X(r,d)\cong H^1(X,E\otimes E^*)$ given for local sections $x_{ij}\otimes \lambda_{ij}$ of $E\otimes E^*$ by $$f(x_{ij}\otimes \lambda_{ij})=\varphi(\sigma^*(x_{ij}))\otimes \sigma^*(\,^t\varphi(\lambda_{ij})).$$ 
Clearly, this involution does not depend on the choice of $\varphi$. \\ 
Given $\eta=(\eta_{ij})\in H^1(X,E\otimes E^*)$, it corresponds to an infinitesimal deformation  $$0\ra E\ra \sr E\ra E\ra 0$$
over $X_\varepsilon$. In fact if we set $$g_{ij}=\phi_i\circ(id+\varepsilon\eta_{ij})\circ\phi_j^{-1},$$ where $\phi_i:E|_{U_i}\ra U_i\times \bb C^r$ are some local trivialisations of $E$, then $\{g_{ij}\}$ are transition functions of $\sr E$ (we will prove this in \S \ref{infstu2}, Lemma \ref{gij} below). \\ 
Now $\eta \in T_E\al U_X^{\sigma,\tau}(r,d)$ if and only if $\sr E$ is $\sigma-$invariant. We can choose $\phi_i$ to be $\sigma-$invariant, hence $\sr E$ is $\sigma-$invariant iff  $\eta$ is invariant with respect to $f$. Thus $$T_E\al U_X^{\sigma,\tau}(r,d)\cong H^1(X,E\otimes E^*)_+.$$
\begin{prop} \label{invariantdimension}
The dimension of the locus of $\sigma-$invariant vector bundles of fixed type $\tau$ is given by $$\text{dim}(\al U_X^{\sigma,\tau}(r,d))=r^2(g_Y-1)+1+\sum_{p\in R}k_p(r-k_p),$$
where $(k_p)_{p\in R}$ are the integers associated to $\tau$.
\end{prop}
\begin{proof}
To calculate the dimension of $H^1(X,E\otimes E^*)_+$ we use Lefschetz fixed point theorem (cf. \cite{AB}), to simplify the notations let $$h^1_\pm=\text{dim}_{\bb C}\left(H^1(X,E\otimes E^*)_\pm\right).$$
We have 
$$\begin{cases}      
h^1_++h^1_-=r^2(g_X-1)+1 &\text{(By Riemann-Roch Formula)}\\
h^1_+-h^1_-=1 -\dfrac{1}{2}\sum_{p\in R}\text{Tr}(f_p) &\text{(By Lefschetz fixed point theorem)}
\end{cases},$$
we have used the fact that $h^0(X,E\otimes E^*)_+=1$ (the identity $E\ra E$ is $\sigma-$invariant).\\
By the very definition, $f_p=\varphi_p\otimes \varphi_p$, it follows that the multiplicity of the eigenvalue $-1$ of $f_p$ is $2k_p(r-k_p)$, hence  $Tr(f_p)= (r-2k_p)^2$, so we have 
$$\begin{cases} 
h^1_++h^1_-=2r^2(g_Y-1)+r^2n+1 \\
h^1_+-h^1_-=1-r^2n+2\sum_{p\in R}k_p(r-k_p)
\end{cases}.$$
It follows $$\text{dim}(\al U_X^{\sigma,\tau}(r,d))=h^1_+= r^2(g_Y-1)+1+\sum_{p\in R}k_p(r-k_p).$$   
\end{proof}
In particular, since $\text{det}:  \al U_X^{\sigma,\tau}(r,0)\ra \text{Pic}^{\sigma,\tilde{\tau}}(X)$ is surjective, where $\tilde{\tau}=\{(-1)^{k_p}\}_{p\in R}$, we have 
\begin{align*}\text{dim}(\al{SU}_X^{\sigma,\tau}(r))&=\text{dim}(\al U_X^{\sigma,\tau}(r,0))-g_Y \\&= (r^2-1)(g_Y-1)+\sum_{p\in R}k_p(r-k_p).
\end{align*}
\begin{rema}\label{maxtype} The dimension of the locus of $\sigma-$invariant vector bundle $\al U_X^\sigma(r,d)$ is the maximum of these dimensions : $$\text{dim}(\al U_X^\sigma(r,d))=\begin{cases} 
r^2(g_Y-1)+n\frac{r^2}{2}+1 &r\equiv 0\mod2 \\
r^2(g_Y-1)+n\frac{r^2-1}{2}+1 &r\equiv 1\mod2 
\end{cases}.$$
These dimensions correspond to the following types (called \emph{maximal types})  $$\ak{MAX}= \begin{cases} 
\left\lbrace \tau=(\varphi_p)_p\mod\pm I_r\,|\; k_p=r/2,\;\forall p\in R \right\rbrace &r\equiv 0\mod2 \\
\left\lbrace\tau=(\varphi_p)_p\mod\pm I_r\,|\;k_p=(r+1)/2\,\text{or}\;k_p=(r-1)/2\right\rbrace \;  &r\equiv 1\mod2 
\end{cases}.$$
In the odd case, the cardinal of $\ak{MAX}$ is  $2^{2(n-1)}$. \\
\end{rema}

\begin{rema}\label{parabolic}
Using the results of Balaji and Seshadri (see \cite{BS}), we can identify the moduli space of stable  $\sigma-$invariant vector bundles of type $\tau$ with the space of stable parahoric $\al G_\tau-$torsors over $Y$
$$\al U_X^{\sigma,\tau}(r,d)\cong \al M_Y(\al G_\tau),$$ for some parahoric Bruhat-Tits group scheme $\al G_\tau$ associated to the type $\tau$. In fact, as we deal with $\text{GL}_r-$bundles, the parahoric group scheme $\al G_\tau$ is of parabolic type, which implies that the moduli of $\sigma-$invariant vector bundles of type $\tau$ is isomorphic to the moduli space of parabolic vector bundles with  parabolic structures, related to $\tau$, at the branch points of $X\ra Y$.
\end{rema}

\subsection{Anti-invariant vector bundles}\label{antiinvariant}
\subsubsection{} A vector bundle $E$ over $X$ that admits an isomorphism $$\psi:\sigma^*E\stackrel{\sim}{\longrightarrow} E^*$$
is called a \emph{$\sigma-$anti-invariant} (or simply \emph{anti-invariant}) vector bundle, where $E^*$ is the dual vector bundle. If $E$ is stable, then this isomorphism is unique up to scalar multiplication. Take an isomorphism $$\psi: \sigma^*E\stackrel{\sim}{\longrightarrow} E^*,$$  by pulling back and taking the transpose we get an isomorphism  $$^t(\sigma^*\psi):\sigma^*E\stackrel{\sim}{\longrightarrow} E^*,$$
so there is a non-zero $\lambda\in \bb C$ such that $^t(\sigma^*\psi)=\lambda \psi$. By applying $\sigma^*$ and taking the transpose on this last equality, we deduce $\lambda^2=1$, thus $\lambda=\pm1$.\\
Denote by $\tilde{\psi}$ the non-degenerated bilinear form canonically associated to  $\psi$: $$\tilde{\psi}:E\otimes\sigma^* E\ra \al O_{ X}.$$  
In fact we obtain $\tilde{\psi}$ as the composition  $$\tilde{\psi}:E\otimes \sigma^*E\stackrel{ id\otimes\psi}{\xrightarrow{\hspace*{1cm}}} E\otimes E^*\stackrel{Tr}{\longrightarrow} \al O_{ X}.$$

\begin{defi} We say that $(E,\psi)$ (or $(E,\tilde{\psi})$) is \emph{$\sigma-$symmetric} (resp. \emph{$\sigma-$alternating}) if $\lambda=1$ (resp.  $\lambda=-1$). We denote by $$\al U_X^{\sigma,+}(r)\subset\al U_X(r,0)\;,\;\;\al U_X^{\sigma,-}(r)\subset\al U_X(r,0)$$ the loci of classes  of stable $\sigma-$symmetric (resp. $\sigma-$alternating) vector bundles $E$.
\end{defi} 
Note that $\al U_X^{\sigma,+}(r)$ and  $\al U_X^{\sigma,-}(r)$ are smooth by Lemma \ref{smoothnesslemma}.
\begin{obse} Let $E$ be a $\sigma-$anti-invariant stable vector bundle of rank $r$ such that $r\equiv 1 \mod 2$ and assume that $\pi:X\ra Y$ is ramified, then $E$ is necessarily $\sigma$-symmetric.
\end{obse} 
Indeed, suppose that $r$ is odd and that there exists a $\psi:\sigma^*E\ra E^*$ which is a $\sigma-$alternating isomorphism. Let $p\in R$ be a ramification point, then $\psi_p$ is a symplectic form on $E_p$, this implies that $\text{dim}_{\bb C}(E_p)=rk(E)$ is necessarily even.\\ 

Let us see the case of line bundles. Consider a line bundle $L$ such that $$\sigma^*L\cong L^{-1}.$$ Because we have $L\otimes \sigma^*L\cong \pi^*\text{Nm}(L)$, it follows that $\pi^*\text{Nm}(L)\cong \al O_X$, hence two cases may happen:
\begin{enumerate}
\item If $\pi$ is ramified, then $\pi^*$ is injective, so $\text{Nm}(L)=\al O_Y$.  
\item If $\pi$ is \'etale, then the kernel of $\pi^*$ is $\{\al O_Y,\Delta\}$, so either $\text{Nm}(L)=\al O_Y$ or $\text{Nm}(L)=\Delta$.  
\end{enumerate}
\begin{lemm}\label{normlinebundle}
If $L$ is a line bundle such that $\text{Nm}(L)=\al O_X$ then $L$ is $\sigma-$symmetric. Assume that $\pi$ is \'etale, then if $\text{Nm}(L)=\Delta$ then $L$ is $\sigma-$alternating.
\end{lemm}
\begin{proof}
The line bundle $L\otimes\sigma^*L$ has a canonical linearisation given by transposition. And the line bundle $\pi^*\text{Nm}(L)$ has the canonical linearisation (which we have called positive in the ramified case). These two linearisations are the same via the isomorphism $$L\otimes \sigma^*L\cong \pi^*\text{Nm}(L).$$
Assume that $\text{Nm}(L)=\al O_Y$, the isomorphism $\sigma^*L\cong L^{-1}$ is indeed a global section of $L\otimes \sigma^*L$, which is unique up to scalar multiplication. Then by Remark \ref{propersubspaces}, we have 
\begin{align*}
H^0(X,L\otimes\sigma^*L)_+&=H^0(X,\pi^*\text{Nm}(L))_+\\ &=H^0(Y,\text{Nm}(L))=\bb C.
\end{align*}
This implies that $L$ is $\sigma-$symmetric. \\
If $\pi$ is \'etale and $\text{Nm}(L)=\Delta$, then it is clear that $L$ is anti-invariant, and again by Remark \ref{propersubspaces} we have
\begin{align*}
H^0(X,L\otimes\sigma^*L)_-&=H^0(X,\pi^*\text{Nm}(L))_-\\ &=H^0(Y,\text{Nm}(L)\otimes \Delta)=\bb C.
\end{align*}
Hence $L$ is $\sigma-$alternating.
\end{proof}

\subsubsection{Infinitesimal  study} \label{infstu2}

We want to identify the tangent spaces to  $\al U_X^{\sigma,+}(r)$ at a point $E$.

 Let $E$ be a $\sigma-$symmetric anti-invariant vector bundle and  $\psi:\sigma^*E\cong E^*$, suppose that $E$ is given by the transition functions $f_{ij}=\varphi_i\circ \varphi_j^{-1}:U_{ij}\ra \text{GL}_r$, where the $\varphi_i:E_{U_i}\stackrel{\sim}{\longrightarrow} U_i\times \bb C^r$ are  local trivialisations of $E$. The covering $\{U_i\}_i$ of $X$ is chosen to be $\sigma-$invariant, i.e. $\sigma(U_i)=U_i$ (to get such covering, just pullback a covering of $Y$ that trivialises both $\pi_*\al O_X$ and $\pi_*E$ over $Y$). Note that we can choose $\{\varphi_i\}$ such that the diagram $$\xymatrix{\sigma^*\varphi_i:\sigma^*E_{U_i}\ar[d]_{\psi}\ar[r]&U_i\times\bb C^r\ar[d]^{\sigma\times I_r}\\ ^t\varphi_i^{-1}:E^*_{U_i}\ar[r]& U_i\times \bb C^r}$$ commutes. Indeed,  by taking an \emph{\'etale} neighbourhood $U$ of each point $x\in X$, such that  $\sigma(U)=U$, we can construct a frame $(e_1,\cdots,e_r)$ of $E|_U$ on which the pairing $\tilde{\psi}:E\otimes \sigma^*E\lra \al O_X$ is represented by the trivial matrix $I_r$. To construct such a frame, we apply the Gram-Schmidt process.  As in this procedure, we need to calculate some square roots, that's the reason why we have to work on the \'etale topology. Moreover, we should mention that if we start with a frame $(u_1,\cdots, u_r)$ near $x$, it may happen that $\tilde{\psi}(u_i\otimes\sigma^*u_i)_x=0$, in this case, we just replace $u_i$ with $u_i+u_j$, for some $j>i$ such that $\tilde{\psi}((u_i+u_j)\otimes\sigma^*(u_i+u_j))_x\not=0$. \\ 
Taking such trivialisations, we get transitions functions $f_{ij}$ such that  $\sigma^* f_{ij}=\,^tf_{ij}^{-1}$. We know that the extension $\sr E$ (which corresponds to some $\eta=\{\eta_{ij}\}\in H^1( X,E\otimes E^*)$) is given by transition functions of the form $$f_{ij}+\varepsilon g_{ij}:U_{ij}\ra \text{GL}_r(\bb C[\varepsilon]).$$
We want to find the relation between these transition functions and $\eta$. First of all, in order that $\{f_{ij}+\varepsilon g_{ij}\}$ represents a $1-$cocycle, we must have the two conditions
          $$\begin{cases} g_{ii}=0 \\ g_{ij}f_{ji}+f_{ij}g_{jk}f_{ki}+f_{ik}g_{ki}=0 \end{cases}.$$
Now given $\eta=\{\eta_{ij}\}\in H^1( X, E\otimes E^*)$, which verifies $\eta_{ii}=0$ and  $$\eta_{ij}+\eta_{jk}+\eta_{ki}=0.$$
Each $\eta_{ij}$ can be seen as local morphism $$\xymatrix{\eta_{ij}:E|_{U_{ij}}\ar[rr]\ar[rd]_{\varphi_j} &&E|_{U_{ij}}\ar[ld]^{\varphi_i} \\ &U_{ij}\times\bb C^r.&}$$
 
 Denote by $g_{ij}=\varphi_i\circ \eta_{ij}\circ \varphi_j^{-1}$, we can rewrite the condition on $\eta$ in the form $$\varphi_{i}^{-1}\circ g_{ij}\circ\varphi_{j}+\varphi_{j}^{-1}\circ g_{jk}\circ\varphi_{k}+\varphi_{k}^{-1}\circ g_{ki}\circ\varphi_{i}=0.$$
Composing by $\varphi_{i}$ from the left and $\varphi_{i}^{-1}$ from the right, we get 
$$g_{ij}\circ\varphi_{j}\circ\varphi_{i}^{-1}+\varphi_{i}\circ\varphi_{j}^{-1}\circ g_{jk}\circ\varphi_{k}\circ\varphi_{i}^{-1}+\varphi_{i}\circ\varphi_{k}^{-1}\circ g_{ki}=0$$
 $$\Leftrightarrow g_{ij} f_{ji}+f_{ij} g_{jk} f_{ki}+f_{ik} g_{ki}=0.$$

\begin{lemm}\label{gij} $f_{ij}+\varepsilon g_{ij}=\varphi_i\circ(id+\varepsilon\eta_{ij})\circ\varphi_j^{-1}$ are transition functions of $\sr E$.
\end{lemm}
\begin{proof}
Let $\eta=\{\eta_{ij}\}\in H^1( X, E\otimes E^*)$, locally the extension $\sr E$ is trivial, that's $$\sr E|_{U_{i\varepsilon}}\cong E|_{U_i}\oplus \varepsilon E|_{U_i},\;\; x\mapsto (\varpi(x), x-s_i\circ\varpi(x)),$$
where $\varpi:\sr E\ra E$ and $s_i$ is a local section of $\varpi$ on the local open set $U_i$, and $U_{i\varepsilon}=U_i\times Spec(\bb C[\varepsilon])$. This isomorphism is $\al O_{ X_\varepsilon}-$linear. \\
Composing with the trivialisation $$\varphi_i+\varepsilon \varphi_i:E|_{U_i}\oplus \varepsilon E|_{U_i}\stackrel{\sim}{\longrightarrow}\al O_{U_i}\oplus \varepsilon\al O_{U_i},$$
we get a trivialisation $$\phi_i:\sr E|_{U_{i\varepsilon}}\stackrel{\sim}{\longrightarrow} \al O_{ U_{i\varepsilon}},$$
given by $$\phi_i=\varphi_i\circ\varpi+\varepsilon\varphi_i(id-s_i\circ\varpi).$$
Remark that $$\phi_i^{-1}= s_i\circ \varphi_i^{-1}-\varepsilon s_i(id-s_i\circ\varpi)s_i\circ\varphi_i^{-1}.$$
So, we calculate the transition functions of $\sr E$ 
\begin{align*}
\phi_i\circ\phi_j^{-1}&=(\varphi_i\circ\varpi+\varepsilon\varphi_i(id-s_i\circ\varpi))(s_j\circ\varphi_j^{-1}-\varepsilon s_j(id-s_j\circ\varpi)s_j\circ\varphi_j^{-1})\;\;\;\;(\text{because }\varepsilon^2=0)\\
&=f_{ij}+\varepsilon(\varphi_i(id-s_i\circ\varpi)s_j\circ\varphi_j^{-1}-\varphi_i\circ\varpi\circ s_j(id-s_j\circ\varpi)s_j\circ\varphi_j^{-1})\\
&= f_{ij} +\varepsilon(\varphi_i(s_j-s_i)\varphi_j^{-1})\\
&= f_{ij}+\varepsilon (\varphi_i\eta_{ij}\varphi_j^{-1})\\
&= f_{ij}+\varepsilon g_{ij}.
\end{align*}
\end{proof}

Now $\eta$ is in the tangent space to $\al U_X^{\sigma,+}(r)$  at $E$ if and only if the corresponding extension $\sr E$ is $\sigma-$symmetric anti-invariant vector bundle on $X_\varepsilon$, where $\sigma$ extended to an involution on $ X_\varepsilon$ by taking $\sigma(\varepsilon)=\varepsilon$.\\
On the transition functions, this means that $$\sigma^*(f_{ij}+\varepsilon g_{ij})=\,^t(f_{ij}+\varepsilon g_{ij})^{-1},$$
which gives\footnote{Recall that $(f+\varepsilon g)^{-1}=f^{-1}-\varepsilon f^{-1}gf^{-1}$ in $\text{GL}_r(\bb C[\varepsilon])$, and $\text{det}(f+\varepsilon g)=\text{det}(f)(1+\varepsilon Tr(f^{-1}g))$.} 
$$\sigma^*f_{ij}=\,^tf_{ij}^{-1},$$ and
\begin{align*}
& \;\;\sigma^*g_{ij}=-\,^tf_{ij}^{-1}\,^t g_{ij}\,^tf_{ij}^{-1} \\
\Leftrightarrow \;\;& \sigma^*\varphi_i\circ \sigma^*\eta_{ij}\circ \sigma^*\varphi_j^{-1} =-\,^t\varphi_{i}^{-1}\circ\,^t\eta_{ij}\circ \,^t\varphi_{j} \\
\Leftrightarrow \;\;& \sigma^*\varphi_i\circ \sigma^*\eta_{ij}\circ \sigma^*\varphi_j^{-1} =-\sigma^*\varphi_{i}\circ\psi^{-1}\circ\,^t\eta_{ij}\circ \psi\circ \sigma^*\varphi_{j}^{-1} \\
\Leftrightarrow \;\;&\sigma^*\eta_{ij} =-\psi^{-1}\circ\,^t\eta_{ij}\circ\psi \\
\Leftrightarrow \;\;&\sigma^*(\eta_{ij}\circ\,^t\psi^{-1}) =-\psi^{-1}\circ\,^t\eta_{ij}=-\,^t(\eta_{ij}\circ\,^t\psi^{-1}) .
\end{align*}

Thus $$\eta\circ\,^t\psi^{-1}\in H^1( X,E\otimes\sigma^*E)_-.$$
where $H^1( X,E\otimes\sigma^*E)_-$ is the proper subspace associated to the eigenvalue $-1$ of the involution of $H^1( X,E\otimes\sigma^*E)$ given by $$\xi\ra \sigma^*(\,^t\xi).$$

 Consider the case of $\al U_X^{\sigma,-}(r)$. Assume that $r$ is even and $\pi$ is ramified. Fix a point $E$ of $\al U_X^{\sigma,-}(r)$. In this case, $\tilde{\psi}$ can be represented with respect to some frame near each fixed point of $X$ by the matrix $$J_r=\begin{pmatrix}0& I_r\\-I_r& 0\end{pmatrix}.$$
Such frame gives a trivialisation $\{\varphi_i\}$  such that $$(\sigma\times J_r)\circ\sigma^*\varphi_i=\,^t\varphi_i^{-1}\circ\sigma^*\psi,$$
so the assocaited transition functions  $\{f_{ij}\}$  verify $$\sigma^* f_{ij}=-J_r\,^tf_{ij}^{-1}J_r.$$ 
It follows that the deformation $\sr E$ is in the tangent space $T_E\al U_X^{\sigma,-}(r)$ if and only if we have 
\begin{align*}
\sigma^*(f_{ij}+\varepsilon g_{ij})&=-J_r\,^t(f_{ij}+\varepsilon g_{ij})^{-1}J_r\\
                                       &=- J_r\,^tf_{ij}^{-1}J_r +\varepsilon J_r \,^tf_{ij}^{-1} \,^tg_{ij} \,^tf_{ij}^{-1} J_r .
\end{align*}
Thus 
\begin{align*}
&\sigma^*\varphi_i\circ\sigma^*\eta_{ij}\circ\sigma^*\varphi_j^{-1}=J_r\,^t\varphi_{i}^{-1}\circ \,^t\eta_{ij}\circ\,^t\varphi_jJ_r\\
&\Leftrightarrow \sigma^*\eta_{ij}=-\psi^{-1}\circ\,^t\eta_{ij}\circ\psi\\
&\Leftrightarrow \sigma^*(\eta_{ij}\circ\,^t\psi^{-1})=\,^t(\eta_{ij}\circ\,^t\psi^{-1}).
\end{align*}

Finally  $$\eta\circ\,^t\psi^{-1}\in H^1( X,E\otimes\sigma^*E)_+.$$

  We have showed so far 
 \begin{theo} With the above notations,  we have 
 \begin{enumerate}
 \item [$(a)$] The tangent space to $\al  U_X^{\sigma,+}(r)$ at a point $E$ is isomorphic to   $H^1( X,E\otimes\sigma^*E)_-$. In particular we have \begin{align*}\text{dim}(\al U_X^{\sigma,+}(r))&=\dfrac{r^2}{2}(g_X-1)+\dfrac{nr}{2}\\&=r^2(g_Y-1)+n\dfrac{r(r+1)}{2}.
\end{align*}
 \item [$(b)$] The tangent space to $\al U_X^{\sigma,-}(r)$ at a point $E$ is isomorphic to  $H^1( X, E\otimes\sigma^*E)_+$. In particular we have \begin{align*}\text{dim}(\al U_X^{\sigma,-}(r))&=\dfrac{r^2}{2}( g_X-1)-\dfrac{nr}{2} \\&=r^2(g_Y-1)+n\dfrac{r(r-1)}{2}.
\end{align*}
 \end{enumerate}
\end{theo}
 \begin{proof}
 We need just to calculate the dimensions. Let $E$ be a $\sigma-$anti-invariant stable vector bundle, denote by $F=E\otimes\sigma^*E$. First we have 
\begin{equation}\label{equ1}
 h^1( X,F)=h^1_++h^1_-=r^2( g_X-1)+1,
\end{equation}  
 where we denote for simplicity $h^0_\pm=h^0( X,F)_\pm$, $h^1_\pm=h^1( X,F)_\pm$.\\
  Let $\varsigma:\sigma^*F\ra F$ be the canonical linearisation which equals to the transposition ($\sigma^*(s\otimes\sigma^* t)\lra t\otimes\sigma^* s$).
  
  Using Lefschetz fixed point formula \cite{AB}, one gets 
 $$h^1_+-h^1_-=h^0_+-h^0_--\sum_{p\in R} \dfrac{\text{Tr}(\varsigma_p)}{\text{det}(id-d_p\sigma)}.$$
 It is clear that $d_p\sigma:T_p X\ra T_p X$ is equal to $-id$ (see Lemma \ref{descentcanonical}), and the trace of the involution $\varsigma_p:F_p\ra F_p$ is equal to $$\text{dim}(F_p)_+-\text{dim}(F_p)_-.$$
 But, $F_p= E_p\otimes E_p=\text{Sym}^2E_p\oplus\bigwedge^2E_p$, and $h^0_+=1$ if $\psi$ is $\sigma-$symmetric, $h^0_-=1$ if $\psi$ is $\sigma-$alternating. Hence

\begin{eqnarray} \label{equ2}
h^1_+-h^1_-&=-\dfrac{1}{2}\left(\sum_{p\in R}\dfrac{r(r+1)}{2}-\dfrac{r(r-1)}{2}\right)+1\nonumber\\
           &=-nr+1 \hspace{1.3cm}\text{if $\psi$ is $\sigma-$symmetric}.\\
h^1_+-h^1_-&=-\dfrac{1}{2}\left(\sum_{p\in R}\dfrac{r(r+1)}{2}-\dfrac{r(r-1)}{2}\right)-1\nonumber\\
           &=-nr-1 \hspace{1.3cm}\text{if $\psi$ is $\sigma-$alternating}.\nonumber
\end{eqnarray}
 From (\ref{equ1}) and (\ref{equ2}), we deduce 
 $$h^1_-=\dfrac{r^2}{2}(g_X-1)+\dfrac{nr}{2}\;\;\;\text{ if $\psi$ est $\sigma-$symmetric.}$$
 $$h^1_+=\dfrac{r^2}{2}(g_X-1)-\dfrac{nr}{2}\;\;\;\text{ if $\psi$ est $\sigma-$alternating.}$$
 The other equalities are consequences of Hurwitz formula.
\end{proof}

Using \cite{BS} Theorem $4.1.6$, we get an identification of the moduli stack $\text{Bun}_X^{\sigma,+}(r)$ (resp.  $\text{Bun}_X^{\sigma,-}(r)$) of pairs $(E,\psi)$, where $E$ is an anti-invariant vector bundle and $\psi$ is a $\sigma-$symmetric (resp. $\sigma-$alternating) isomorphism $\sigma^*E\ra E^*$, with the moduli stack $\text{Bun}_Y(\al G)$ of $\al G-$torsors over $Y$, where $\al G=\pi_*^{\tilde{\sigma}}(\text{GL}_r)$ is the invariant direct image of the constant group scheme $\text{GL}_r$ over $X$, where $\tilde{\sigma}$ acts on $\text{GL}_r$ by $$g\ra\, ^t\sigma(g)^{-1}\;\;(\text{resp. } J_r\,^t\sigma(g)^{-1}\,J_r^{-1}, \text{ if $r$ is even}).$$
We can see, using the description of parahoric  subgroups given in \cite{BS}, that $\al G$ is not a parahoric Bruhat-Tits group scheme (in the sense given in \cite{BS}). Moreover, it is not generically constant. This gives a new example of interesting moduli spaces.

\section{The Hitchin system for anti-invariant vector bundles}
For $s\in W$, we denote by  $q:\tilde{X}_s \ra X$ the associated spectral cover of $X$, and by $\tilde{S}=Ram(\tilde{X}_s /X)$ its ramification divisor.\\
Fix the positive linearisations on $K_X$ and $\al O_X$ (see Remark \ref{propersubspaces}). Recall that this linearisation equals $id$ over the ramification points. We denote these linearisations by   $$\eta:\sigma^*K_X\ra K_X,\;\,\;\nu:\sigma^*\al O_X\ra \al O_X.$$
The linearisation $\eta$ induces an involution on the space of global sections of $K_X^i$ for each $i\geqslant1$, we define $$W^{\sigma,+}=\bigoplus_{i=1}^r H^0(X,K_X^i)_{+}.$$
\begin{prop} \label{descentram} Consider an $r-$tuple of global  sections $s=(s_1,\cdots,s_r)\in W^{\sigma,+}$  and let $\tilde{X}_s $ be the associated spectral curve over $X$. Then the involution $\sigma:X\ra X$ lifts to an involution $\tilde{\sigma}$ on $\tilde{X}_s $ and $\al O(\tilde{S})$ descends to $\tilde{Y}_s:=\tilde{X}_s /\tilde{\sigma}$. 
\end{prop}
\begin{proof} 
We have an isomorphism  $$\al O_X\oplus K_X^{-1}\stackrel{\,^t\nu\otimes\,^t\eta}{\xrightarrow{\hspace*{1cm}}}\sigma^*(\al O_X\oplus K_X^{-1}),$$ which induces an involution $\bar{\sigma}$ on $\bb S=\bb P(\al O_X\oplus K_X^{-1})$. Let $\tilde{X}_s \subset\bb P(\al O_X\oplus K_X^{-1})$ be the spectral curve associated to $s$. \\
Recall that  $y$ is defined to be the identity section of $\bar q_*\al O(1)\cong \al O_X\oplus K_X^{-1}$, therefore it is $\bar{\sigma}-$invariant. The section $x$ is by definition the canonical section of $\bar q^*K_X\otimes \al O(1)$. In fact it can be seen as the canonical section of $\bar q^*K_X\ra |K_X|$, where $|K_X|$ is the total space of $K_X$. Hence $x$ is invariant with respect to the positive linearisation.\\
As by definition $\eta^{\otimes k}(\sigma^*(s_k))=s_k$,  we deduce that $$\bar{\sigma}((\bar q^*s_k)y^kx^{r-k})=(\bar q^*s_k)y^kx^{r-k}.$$ Thus the section defining $\tilde{X}_s $ 
\begin{equation*} x^r+(\bar q^*s_1)yx^{r-1}+\cdots+ (\bar q^*s_r)y^r \in H^0(\bb S,\bar{q}^*K_X^r\otimes \al O(r))\end{equation*} 
 is  $\bar{\sigma}-$invariant. 
Hence $\bar\sigma(\tilde{X}_s )=\tilde{X}_s $, so $\bar{\sigma}$ induces an involution on $\tilde{X}_s $ which we denote by $\tilde{\sigma}$. \\
Remark that $\bar\sigma$ acts trivially on the fibers of $\bar q:\bb S\ra X$ over the ramification points of $\pi:X\ra Y$. Thus the ramification locus of $\tilde{\sigma}$ is $q^{-1}(R)$.\\

By Hurwitz formula we have $\al O(\tilde{S})=K_{\tilde{X}_s }\otimes q^*K_X^{-1}$. We know by Lemma \ref{descentcanonical} that $K_{\tilde{X}_s }$ (resp. $K_X$)  descends to $\tilde{Y}_s$ (resp. $Y$). Moreover,  $K_{\tilde{X}_s }=\tilde{\pi}^*K_{\tilde{Y}}(\tilde{R})$ (resp. $K_X=\pi^* K_Y(R)$), where  $\tilde{R}=Ram(\tilde{X}_s/\tilde{Y}_s)$, and we have used the notation of the commutative diagram 
$$\xymatrix{\tilde{X}_s \ar[r]^q \ar[d]_{\tilde{\pi}} & X \ar[d]^\pi \\ \tilde{Y}_s \ar[r]^{\tilde{q}} &  Y,}$$
since $\al O(\tilde{R})=q^*\al O(R)$, it follows 
\begin{align*}
\al O(\tilde{S})&=K_{\tilde{X}_s }\otimes q^*K_X^{-1}\\
                &=\tilde{\pi}^*K_{\tilde{Y}_s}\otimes q^*(\pi^*K_Y^{-1})\otimes\al O(\tilde{R})\otimes q^*\al O(-R) \\
                &= \tilde{\pi}^*\left(K_{\tilde{Y}_s}\otimes\tilde{q}^*K_Y^{-1}\right).
\end{align*} 
In particular, by Hurwitz formula,  $K_{\tilde{Y}}\otimes\tilde{q}^*K_Y^{-1}=\al O(S)$, where $S=Ram(\tilde{Y}_s/Y)$, hence $\al O(\tilde{S})=\tilde{\pi}^*\al O(S)$.
\end{proof}

We keep the notations of the last proposition hereafter.

\begin{rema}\label{Ytildespectral} Remark that for $s\in W^{\sigma,+}$,  $\tilde{Y}_s$ is a spectral cover of $Y$ associated to the line bundle $L=K_Y\otimes \Delta$, because the sections $s_i$ descend to $Y$.
\end{rema}

\begin{lemm} \label{serre}
Let $F$ be a $\sigma-$linearised vector bundle, and consider the positive linearisation on $K_X$. Then the Serre duality isomorphism  $$H^1(X,F^*)\stackrel{\sim}{\longrightarrow}H^0(X, F\otimes K_X)^*$$ is anti-equivariant with respect to the induced involutions on the two spaces.
\end{lemm}
\begin{proof}
 If $F$ is a $\sigma-$linearised vector bundle, we have an equivariant perfect pairing: $$H^0(X,F)\otimes H^1(X,F^*\otimes K_X)\ra H^1(X,K_X)\stackrel{\sim}{\longrightarrow}\bb C.$$
As the fixed linearisation is the positive one, it follows by Remark \ref{propersubspaces} that 
\begin{align*}
H^1(X,K_X)_-&=H^1(X,\pi^*(K_Y\otimes\Delta))_+ \\ &= H^1(Y,K_Y\otimes \Delta\otimes \Delta^{-1}) \\&= H^1(Y,K_Y)=\bb C.
\end{align*}
So $$H^1(X,K_X)=H^1(X,K_X)_-. $$
Since the above pairing is equivariant, we get the result.
\end{proof}
As a direct consequence, one gets an isomorphism
\begin{align*}
 T^*\al U_X^{\sigma,+}(r)\stackrel{\sim}{\lra} H^0(X,E\otimes \sigma^*E\otimes K_X)_+,\\
 T^*\al U_X^{\sigma,-}(r)\stackrel{\sim}{\lra} H^0(X,E\otimes \sigma^*E\otimes K_X)_-.
\end{align*}

We denote by $\sr H_i$ the $i^{th}$ component of the Hitchin map $$\sr H_E:H^0(X,E\otimes \sigma^* E\otimes K_X)\ra W.$$
\begin{prop}  \label{Hitchinmapequivarience}Let $E$ be $\sigma-$anti-invariant stable vector bundle, and let $\psi:\sigma^*E\cong E^*$ be an isomorphism.
\begin{enumerate}
\item If $\psi$ is $\sigma-$symmetric, then $\sr H_i$ induces a map $$\sr H_i:H^0(X,E\otimes \sigma^*E\otimes K_X)_+\ra H^0(X,K_X^i)_{+}.$$
\item If $\psi$ is $\sigma-$alternating, then $\sr H_i$ induces a map $$\sr H_i:H^0(X, E\otimes \sigma^*E\otimes K_X)_-\ra H^0(X,K_X^i)_{+}.$$
\end{enumerate}
\end{prop}
\begin{proof}Let $\ak t$ be the canonical linearisation on $E\otimes \sigma^*E$ given by the transposition, then the linearisation $\ak t\otimes \eta$ on $E\otimes \sigma^*E\otimes K_X$ induces an involution on $H^0(X,E\otimes \sigma^*E\otimes K_X)$ which we denote by $f$.\\
Let $\phi\in H^0(X,E\otimes \sigma^*E\otimes K_X)$, locally we can write $\phi=\sum_ks_k\otimes \sigma^*(t_k)\otimes \alpha_k$, where $\alpha_k$ (resp. $s_k$, $t_k$) are local sections of $K_X$ (resp. $E$). We can see the section $\phi$ as a map $E\ra E\otimes K_X$ which is defined locally by $$x\lra \phi(x)=\sum_k\gen{\psi(\sigma^*(t_k)),x}s_k\otimes \alpha_k.$$
Thus $\bigwedge^i\phi$ is defined locally by 
\begin{align*}
\bigwedge^i\phi(x_1\wedge\cdots\wedge x_i)&=i!\phi(x_1)\wedge\cdots\wedge\phi(x_i)\\
                                          &=i!\left(\sum_{k_1}\gen{\psi(\sigma^*(t_{k_1})),x_1} s_{k_1}\otimes\alpha_{k_1}\right)\wedge\cdots\wedge\left(\sum_{k_i}\gen{\psi(\sigma^*(t_{k_i})),x_i}s_{k_i}\otimes\alpha_{k_i} \right)\\
                                          &=i!\sum_{k_1,\dots,k_i} \gen{\psi(\sigma^*(t_{k_1})),x_1}\cdots\gen{\psi(\sigma^*(t_{k_i})),x_i}s_{k_1}\wedge\cdots\wedge s_{k_i} \otimes \left(\bigotimes_{j=1}^i\alpha_{k_j}\right)\\
                                          &=i!\sum_{k_1<\dots<k_i} \text{det}\left(\gen{\psi(\sigma^*(t_{k_j})),x_l}\right)_{j,l}s_{k_1}\wedge\cdots\wedge s_{k_i}\otimes \left( \bigotimes_{j=1}^i\alpha_{k_j}\right)\\
                                          &=i!\sum_{k_1<\dots<k_i} \gen{(\bigwedge^i\psi)(\sigma^*(t_{k_1})\wedge\cdots\wedge\sigma^*(t_{k_i})),x_1\wedge\cdots\wedge x_i} s_{k_1}\wedge\cdots\wedge s_{k_i}\otimes\left(\bigotimes_{j=1}^i\alpha_{k_j}\right).
\end{align*}
For the last equality, we use the canonical isomorphism $\bigwedge^k E^*\cong(\bigwedge^kE)^*$ given by the determinant. It follows that (locally) we have $$\bigwedge^i\phi=i!\sum_{k_1<\dots<k_i}s_{k_1}\wedge\cdots\wedge s_{k_i}\otimes\sigma^*(t_{k_1})\wedge\cdots\wedge\sigma^*(t_{k_i})\otimes \bigotimes_{j=1}^i\alpha_{k_j}.$$
\begin{enumerate}
\item Suppose that $\psi$ is $\sigma-$symmetric, thus for any local section  $s$ and $t$ of $E$, one has  $$\gen{\psi(\sigma^*(t)),s}=\nu(\sigma^*\gen{\psi(\sigma^*(s)),t}).$$ 
Hence 
\begin{align*}
\sr H_i(f(\phi))&=(-1)^i \text{Tr}(\,^t(\sigma^*(\bigwedge^i\phi)))\\
                     &=(-1)^i i!\sum_{k_1<\dots<k_i} \gen{\bigwedge^i\psi(\sigma^*(s_{k_1})\wedge\cdots\wedge\sigma^*(s_{k_i})),t_{k_1}\wedge\cdots\wedge t_{k_i}}\bigotimes_{j=1}^i\eta(\sigma^*(\alpha_{k_j}))\\
                     &=(-1)^i i!\sum_{k_1<\dots<k_i} \text{det}\left(\gen{\psi(\sigma^*(s_{k_l})),t_{k_{l'}}}\right)_{1\leqslant l,l'\leqslant i}\bigotimes_{j=1}^i\eta(\sigma^*(\alpha_{k_j}))\\
                     &=(-1)^i i!\sum_{k_1<\dots<k_i} \nu\left(\sigma^*\gen{\bigwedge^i\psi(\sigma^*(t_{k_1})\wedge\cdots\wedge\sigma^*(t_{k_i})),s_{k_1}\wedge\cdots\wedge s_{k_i}}\right)\bigotimes_{j=1}^i\eta(\sigma^*(\alpha_{k_j}))\\
                     &=\eta^{\otimes i}(\sigma^*(\sr H_i(\phi))).
\end{align*}  
Thus, if $f(\phi)=\phi$, then $ \eta^{\otimes i}(\sigma^*(\sr H_i(\phi)))=\sr H_i(\phi)$.
\item If $\psi$ is $\sigma-$alternating, thus $$\gen{\psi(\sigma^*(t)),s}=-\nu(\sigma^*\gen{\psi(\sigma^*(s)),t}),$$ 
by the above calculation, it follows that   
$$\sr H_i(f(\phi))=(-1)^i\eta^{\otimes i}(\sigma^*(\sr H_i(\phi))).$$
On the other hand, it is  clear that $$\sr H_i(-\phi)=(-1)^i\sr H_i(\phi),$$ 
so if  $f(\phi)=-\phi$, then $$\eta^{\otimes i}(\sigma^*(\sr H_i(\phi)))=\sr H_i(\phi).$$
\end{enumerate}
\end{proof}

We claim that $\text{dim}(W^{\sigma,+} )=\text{dim}(\al U_X^{\sigma,+}(r))$. Indeed we have 
\begin{align*}
H^0(X,K_X^i)&\cong H^0(Y,\pi_*K_X^i)\\
     & = H^0(Y,K_Y^i\otimes \Delta^i)\oplus H^0(Y,K_Y\otimes \Delta^{i-1}). 
\end{align*}
As the fixed linearisation on $K_X$ is the positive one, by Remark \ref{propersubspaces}, we obtain 
 $$H^0(X,K_X^i)_+\cong H^0(Y,K_Y^i\otimes \Delta^i),$$
hence $$h^0(X,K_X^i)_+=(2i-1)(g_Y-1)+in,$$ it follows that 
\begin{align*}
\text{dim}(W^{\sigma,+} )&=\sum_{i=1}^r(2i-1)(g_Y-1)+in \\ 
&=r^2(g_Y-1)+\dfrac{r(r+1)}{2}n.
\end{align*}
\begin{rema}
 One can use Lefschetz fixed point theorem (\cite{AB}) to calculate the dimension of $W^{\sigma,+}$.
\end{rema}

To study the irreducibility of $\al U_X^{\sigma,+}(r)$ and $\al U_X^{\sigma,-}(r)$, we will use the notion of \emph{very stable} vector bundles, which has been introduced in \cite{LG}, but we focus just on stable vector bundles. Let $E$ be a stable  vector bundle, and let $\phi:E\ra E\otimes K_X$ be a Higgs field. We say that $\phi$ is \emph{nilpotent} if the composition of the maps $$E\stackrel{\phi}{\lra}E\otimes K_X\stackrel{\phi\otimes id}{\lra}E\otimes K_X^2\ra\cdots\ra E\otimes K_X^{r-1}\stackrel{\phi\otimes id}{\lra}E\otimes K_X^r$$ vanishes. 
\begin{defi} We say that a vector bundle $E$ is \emph{very stable} if $E$ has no nilpotent Higgs field other than $0$. 
\end{defi}
If $E$ is a very stable vector bundle, then the  Hitchin morphism $$\sr H_E:H^0(X,E\otimes E^*\otimes K_X)\ra W$$ is dominant. Indeed, by the very definition, $\sr H_E^{-1}(0)=\{0\}$, but the two spaces have the same dimension, this implies that $\sr H_E$ is dominant.

One of the main results of \cite{LG} is that the locus of very stable vector bundles is an open dense  subscheme of the moduli space of vector bundles.

\begin{defi} We say that a $\sigma-$symmetric (resp. $\sigma-$alternating) anti-invariant vector bundle $E$ is \emph{very stable} if $E$ has no nilpotent Higgs field $$\phi\in H^0(E\otimes \sigma^*E\otimes K_X)_+ \;\;(\text{ resp. } \phi\in H^0(E\otimes \sigma^*E\otimes K_X)_-)$$  other than $0$. 
\end{defi}

Let $T^*\al U_X(r,0)$ be the cotangent bundle of  $\al U_X(r,0)$. This bundle is invariant with respect to the involution $E\ra \sigma^*E^*$ on $\al U_X(r,0)$. In fact, this is true more generally for any variety $Z$ with an involution $\tau$. To see this consider the differential of $\tau$, it gives a linear isomorphism $$d\tau:TZ\lra \tau^*TZ,$$ 
but $\tau^2=id_Z$, this implies that $d\tau\circ\tau^*d\tau=id$. Thus $d\tau$ is a linearisation on $TZ$, hence $d\tau^t$ is a linearisation on $T^*Z$.

In particular, in our case, the involution $E\ra \sigma^*E^*$ of  $\al U_X(r,0)$ lifts to an involution on $T^*\al U_X(r,0)$. If we identify $T^*_E\al U_X(r,0)\cong H^0(X,E\otimes \sigma^*E\otimes K_X)$ using $\psi:\sigma^*E\cong E^*$ and Serre duality, then this lifting is the involution $f$ on $H^0(X,E\otimes\sigma^*E\otimes K_X)$ given by $f(\phi)=\,^t(\sigma^*\phi)$ in the $\sigma-$symmetric case, and $f(\phi)=-\,^t(\sigma^*\phi)$ in the $\sigma-$alternating case.\\  Moreover, by \S \ref{infstu2}, the fixed locus of this involution is  the cotangent bundle $T^*\al U_X^{\sigma,+}(r)$ (resp. $T^*\al U_X^{\sigma,-}(r)$). Hence we can consider both $T^*\al U_X^{\sigma,\pm}(r)$ as closed subspaces of $T^*\al U_X(r,0)$. Moreover, the tautological symplectic form on $T^*\al U_X(r,0)$ restricts to the tautological forms on $T^*\al U_X^{\sigma, \pm}(r)$.

Following the notations of \cite{LG}, let $\Lambda_{X,r}\subset T^*\al U_X(r,0)$ the nilpotent cone, that's the set of $(E,\phi)$ with $\phi$ nilpotent Higgs field. Set $$\Lambda_{X,r}^{\sigma,+}=\Lambda_{X,r}\cap T^*\al U_X^{\sigma,+}(r)\,,\;\Lambda_{X,r}^{\sigma,-}=\Lambda_{X,r}\cap T^*\al U_X^{\sigma,-}(r).$$

\begin{theo}\label{denseverystable}
The nilpotent cone $\Lambda_{X,r}^{\sigma,+}$ (resp. $\Lambda_{X,r}^{\sigma,-}$) is Lagrangian in $T^*\al U_X^{\sigma,+}(r)$ (resp. $T^*\al U_X^{\sigma,-}(r)$). In particular the locus of very stable anti-invariant vector bundles is dense in $\al U_X^{\sigma,+}(r)$ (resp. $\al U_X^{\sigma,-}(r)$)
\end{theo}
\begin{proof}We prove the $\sigma-$symmetric case, the $\sigma-$alternating is absolutely the same. If $V$ is symplectic space, then the restriction of a Lagrangian subspace $L\subset V$ to a symplectic subspace $F\subset V$ is an isotropic subspace of $F$, this implies that $\Lambda_{X,r}^{\sigma,+}$ is an isotropic subspace of $T^*\al U_X^{\sigma,+}(r)$, in particular its dimension is at most  $\text{dim}(\al U_X^{\sigma,+}(r))$. But it is clear that $\al U_X^{\sigma,+}(r)\subset \Lambda_{X,r} ^{\sigma,+}$, by seeing any anti-invariant vector bundle $E$ as the trivial pair $(E,0)\in \Lambda_{X,r}^{\sigma,+}$. This implies that $$\text{dim}(\Lambda_{X,r}^{\sigma,+})=\frac{1}{2}\text{dim}(T^*\al U_X^{\sigma,+}(r)),$$
hence $\Lambda_{X,r}^{\sigma,+}$ is Lagrangian of $T^*\al U_X^{\sigma,+}(r)$.
\end{proof}

\subsection{$\sigma-$symmetric case}
\subsubsection{The ramified case}Suppose that $\pi:X\ra Y$ is ramified and denote $m=r(r-1)(g_X-1).$ Recall that $\text{deg}(\tilde{S})=2m$, where $\tilde{S}=Ram(\tilde{X}_s/X)$. We fix the positive linearisation on $\al O(\tilde{S})$. 

For general $s\in W^{\sigma,+}$, consider the subvariety $\al P^+\subset \text{Pic}^m(\tilde{X}_s )$  of isomorphism classes of line bundles $L$ such that $$\widetilde{\text{Nm}}(L)\cong \al O_{\tilde{Y}_s}(S), $$ where $S=Ram(\tilde{Y}_s/ Y)$, and $\widetilde{\text{Nm}}:\text{Pic}^0(\tilde{X}_s )\ra \text{Pic}^0(\tilde{Y}_s)$ the norm map attached to $\tilde{\pi}:\tilde{X}_s\ra \tilde{Y}_s$.\\  
By Proposition \ref{descentram}, $\tilde{\pi}^*\al O_{\tilde{Y}_s}(S)=\al O_{\tilde{X}_s}(\tilde{S})$, it follows that for each $L\in \al P^+$, we have $$\tilde{\sigma}^*L\cong L^{-1}(\tilde{S}).$$ In particular any $M\in \al P^+$ gives by tensor product an isomorphism $$\al P^+\stackrel{\sim}{\longrightarrow}\text{Prym}(\tilde{X}_s /\tilde{Y}_s).$$

\begin{lemm} \label{dimprym} For general $s\in W^{\sigma,+}$, we have 
$$\text{dim}(\al P^+)= \text{dim}(\al U_X^{\sigma,+}(r)).$$
\end{lemm}
\begin{proof}
For general $s\in W^{\sigma,+}$, the curve $\tilde{X}_s $ is smooth and its genus is $g_{\tilde{X}_s}=r^2(g_X-1)+1$. Indeed, by Lemma \ref{smoothlemma}, and because $W^{\sigma,+}$ is irreducible, it suffices to prove that there exists an $s\in W^{\sigma,+}$ such that $\tilde{X}_s$ is smooth. To do so, take $s_r\in H^0(X,K_X^r)_+$ be a general section that has just simple roots which are different from the ramification points (i.e. outside a finite union of hyperplanes of sections vanishing at points of $R$). This is possible because the hyperplane $H^0(K_X^r(-p))\subset H^0(K_X^r)$ containes $H^0(K_X^r)_-$, so necessarly it does not contain $H^0(K_X^r)_+$, and this for every $p\in R$. Then using the proof of Proposition \ref{smoothlemma}, we deduce that the spectral curve attached to $s=(0,\cdots,0,s_r)$ is smooth. 

 Moreover, if $\tilde{X}_s $ is smooth then, by Lemma \ref{smoothnesslemma}, we deduce that $ \tilde{R}$ has no multiple points (i.e. reduced divisor).
Furthermore we have $\text{deg}(\tilde{R})=2rn$,  so we get 
\begin{align*}
\text{dim}(\al P^+)&=g_{\tilde{X}_s }-g_{\tilde{Y}_s}& \\ &= \dfrac{1}{2}(g_{\tilde{X}_s }-1+rn) & \text{ (by Riemann-Roch)} \\ &= \dfrac{r^2}{2}(g_X-1)+\dfrac{rn}{2}. &
\end{align*}
\end{proof}
\begin{rema} We can calculate $\text{dim}(\al P^+)$ using the fact that $\tilde{Y}_s$ is spectral curve over $Y$.
\end{rema}

\begin{theo}\label{main1}Suppose that $\pi:X\ra Y$ is ramified. Then for general $s\in W^{\sigma,+}$, the rational pushforward map $$q_*:\al P^+\dashrightarrow \al U_X^{\sigma,+}(r)$$  is dominant. In particular  $\al U_X^{\sigma,+}(r)$  is irreducible.
\end{theo}
\begin{proof}
First, by the  duality for finite flat morphisms, this map is well defined, more precisely, one has 
\begin{align*}
 \sigma^* q_*L&\cong  q_*(\tilde{\sigma}^*L)\\
                        &\cong  q_*(L^{-1}\otimes \al O(\tilde{S}))\\
                        &\cong  (q_*L)^*,
\end{align*} 
the last isomorphism is the duality for finite flat morphisms (see for example \cite{HA}, Ex. III.6.10). 
The isomorphism $\psi:\sigma^*(q_*L)\ra (q_*L)^*$ is defined using the pairing $\tilde{\psi}:q_*L\otimes\sigma^*(q_*L)\ra \al O_X$, which is defined as follows: for $v,w\in H^0(U,q_*L)=H^0(q^{-1}(U),L)$, we put $$\tilde{\psi}(v\otimes\tilde{\sigma}^*w)\cdot \xi=\gen{v,\tilde\sigma^*w},$$ where $\xi\in H^0(\tilde{X}_s,\al O(\tilde{S}))$ is the canonical section equals the derivative of $q:\tilde{X}_s\ra X$, and $\gen{,}:L\otimes\tilde{\sigma}^*L\ra\al O(\tilde{S})$ is an isomorphism. This last isomorphism is $\tilde{\sigma}-$symmetric for the positive linearisation on $\al O(\tilde{S})$. Indeed, it is a global section of  $\tilde{\sigma}^*L^{-1}\otimes L^{-1}\otimes\al O(\tilde{S})$ ($\cong \al O_{\tilde{X}_s}$) and we have 
\begin{align*} 
H^0(\tilde{X}_s,\tilde{\sigma}^*L^{-1}\otimes L^{-1}\otimes\al O(\tilde{S}))_+&=H^0(\tilde{X}_s,\tilde{\pi}^*\widetilde{\text{Nm}}(L^{-1})\otimes\tilde{\pi}^*(\al O(S)))_+\\&=H^0(\tilde{X}_s,\tilde{\pi}^*(\al O(-S))\otimes\tilde{\pi}^*(\al O(S)))_+\\ &= H^0(\tilde{Y}_s, \al O_{\tilde{Y}_s})\\&=\bb C .
\end{align*}
Further $\xi$ is $\tilde{\sigma}-$invariant global section of $\al O(\tilde{S})$ with respect to the positive linearisation. Hence $\psi$ is $\sigma-$symmetric. 

 Conversely, given a $\sigma-$symmetric anti-invariant stable vector bundle $E$ and $\phi\in H^0(X,E\otimes \sigma^*E\otimes K_X)_+$ such that $\sr H_E(\phi)=s$, then the corresponding line bundle over $\tilde{X}_s $ is in $\al P^+$. To see this, consider the exact sequence (see \cite{BNR}, Remark 3.7) 
\begin{equation}\label{seq1}
0\ra L(-\tilde{S})\lra q^*(E)\xrightarrow{q^*\phi-x} q^*(E\otimes K_X)\lra L\otimes q^*K_X\ra 0.
\end{equation} 
By taking the dual,  pulling-back by $\tilde{\sigma}$ and than taking the tensor product by $\tilde{\sigma}^*q^*K_X$, we get the exact sequence
\begin{equation} \label{seq2}
0\ra \tilde{\sigma}^*(L^{-1})\ra \tilde{\sigma}^*q^*(E^*)\xrightarrow{\tilde{\sigma}^*(\,^t(q^*\phi))-\tilde{\sigma}^*x}\tilde{\sigma}^*( q^*(E^*\otimes K_X))\ra \tilde{\sigma}^*(L^{-1}(\tilde{S})\otimes q^*K_X)\ra 0 \end{equation}
Since $\phi$ is invariant: $^t(\sigma^*\phi)=\phi$, the middle maps of the exact sequences (\ref{seq1}) and (\ref{seq2}) are identified using the isomorphism $\psi:\sigma^*E\ra E^*$, hence they have  isomorphic kernels. This implies that $L(-\tilde{S})\cong \tilde{\sigma}^*(L^{-1}).$
Thus $L\in \al P^+$.\\ 

Moreover, by Lemma \ref{dimprym} we deduce that whenever $\tilde{X}_s $ is smooth we have  $\text{dim}(\al P^+)=\text{dim}(\al U_X^{\sigma,+}(r))$.\\ 
Now, if $E\in \al U_X^{\sigma,+}(r)$ is very stable then the map $$\sr H_E:H^0(X,E\otimes\sigma^*E\otimes K_X)_+\ra W^{\sigma,+}$$ is dominant, it follows that the map$$\Pi:T^*\al U_X^{\sigma,+}(r)\lra \al U_X^{\sigma,+}(r) \times W^{\sigma,+}$$ is dominant too, because the locus of very stable vector bundles is dense inside $\al U_X^{\sigma,+}(r)$ by Theorem \ref{denseverystable}. In particular, fixing a general $s\in W^{\sigma,+}$, we obtain a dominant morphism $$\sr H^{-1}(s)\lra \al U_X^{\sigma,+}(r),$$ 
where $\sr H:T^*\al U_X^{\sigma,+}(r)\ra W^{\sigma,+}$ is the Hitchin morphism. But, by Proposition \ref{bnrprop} and what we have said above, we deduce that   $(E,\phi)\in\sr H^{-1}(s)$ if and only if $E\cong q_*L$ for some $L\in \al P^+$. It follows that the rational map $$q_*:\al P^+\dashrightarrow \al U_X^{\sigma,+ }(r)$$ is dominant. 

As $\tilde{R}=q^{-1}(R)$, one deduces that $\tilde{X}_s \ra \tilde{Y}_s$ is ramified. This implies the connectedness of $\al P^+$, hence $\al U_X^{\sigma,+}(r)$ is irreducible.
\end{proof}

\subsubsection{The \'etale case}Assume that the cover $\pi:X\ra Y$ is \'etale. In this case, any $\sigma-$invariant vector bundle over $X$ descends to $Y$ by Kempf's Lemma. In particular, we have $$K_X=\pi^*(K_Y\otimes \Delta)=\pi^*K_Y.$$ The linearisation on $K_X$ attached to $K_Y\otimes \Delta$ is called the positive linearisation. Recall that Serre duality is anti-equivariant for this linearisation.\\ 
Remark that $\al O(\tilde{S})=\tilde{\pi}^*\al O(S)=\tilde{\pi}^*(\al O(S)\otimes \tilde{\Delta})$, where $\tilde{\Delta}=\text{det}(\tilde{\pi}_*\al O_{\tilde{X}_s})^{-1}$. We fix the linearisation on $\al O(\tilde{S})$ attached to the $\al O(S)$ and we continue calling it the positive linearisation.

\begin{theo} \label{etalecase}
Suppose that $\pi:X\ra Y$ is \'etale, then  the pushforward rational map $q_*$ induces a dominant map $$q_*:\al P^+\dashrightarrow \al U_X^{\sigma,+}(r).$$ In particular $\al U_X^{\sigma,+}(r)$ has two connected components.
\end{theo}
\begin{proof}
Clearly $\tilde{X}_s \ra\tilde{Y}_s$ is \'etale if and only if $X\ra Y$ is. Hence $\al P^+$ has two connected components. We show that it is impossible to produce the same stable vector bundle $E$ as the direct image of two line bundles from the two connected components of $\al P^+$. To see this assume that we have  $L$ and $L'$, two line bundles each from a connected component of $\al P^+$, such that $q_*L\cong q_*L'\in \al U_X^{\sigma,+}(r)$. Let $M$ be a line bundle on $\tilde{X}_s $ such that $M$ descends to $\tilde{Y}_s$ and $\widetilde{\text{Nm}}(M)=\al O(S)$, in particular $M^2\cong \al O(\tilde{S})$. Let  $\kappa$ be an even theta characteristic on $\tilde{X}_s $ such that $M^{-1}\otimes\kappa$  is the pullback of a theta characteristic $\kappa'$ on $Y$, i.e $M^{-1}\otimes\kappa=q^*\left(\pi^*(\kappa')\right)$, note that  such a pair $(M,\kappa)$ exists by Lemma \ref{pullbackkappa} below. Then, by \cite{BL}, Theorem 12.6.2, we know that $$h^0(\tilde{X}_s ,L\otimes M^{-1}\otimes\kappa)\equiv 0\mod 2,\;\;h^0(\tilde{X}_s ,L'\otimes M^{-1}\otimes\kappa)\equiv 1\mod 2$$ 
Using the projection formula, this gives $$h^0(X,q_*L\otimes \pi^*\kappa')\equiv 0\mod2\; \text{ and }\; h^0(X,q_*L'\otimes \pi^*\kappa')\equiv 1\mod2$$ a contradiction.

Moreover, if  $E\in \al U_X^{\sigma,+}(r)$, then the associated line bundle $L$ over $\tilde{X}_s$ constructed in the proof of Theorem \ref{main1}, verifies $\tilde{\sigma}^*L\cong L^{-1}(\tilde{S})$. Since $\tilde{\pi}:\tilde{X}_s\ra \tilde{Y}_s$ is \'etale, it follows that either $\widetilde{\text{Nm}}(L)=\al O_{\tilde{Y}_s}(S)$, or $\widetilde{\text{Nm}}(L)=\al O_{\tilde{Y}_s}(S)\otimes \tilde{\Delta}$. But $\psi$ induces a $\tilde{\sigma}-$symmetric isomorphism $\tilde{\sigma}^*L\ra L^{-1}(\tilde{S})$, and because we have fixed the positive linearisation on $\al O_{\tilde{X}_s}(\tilde{S})$, it follows that 
$\widetilde{\text{Nm}}(L)=\al O_{\tilde{Y}_s}(S)$. So $L\in \al P^+$. \\
Now the image of the rational map  $q_*:\al P^+\ra\al U_X^{\sigma,+}(r)$ has two connected components, which are dense, and by Mumford \cite{M2}, the map $$E\ra h^0(X,\pi_*E\otimes \kappa')$$ is constant under deformation of $E$. Hence $\al U_X^{\sigma,+}(r)$ can't be irreducible. It follows that it has two connected components. 
\end{proof}

\begin{lemm} \label{pullbackkappa} Suppose that $\pi:X\ra Y$ is \'etale and $\tilde{X}_s$ is smooth. Then there exists an even theta characteristic $\kappa$ on $\tilde{X}_s $, and a line bundle $M$ that descends to $\tilde{Y}_s$ and  verifies $M^2\cong \al O(\tilde{S})$, such that $M^{-1}\otimes\kappa$ descends to a theta characteristic on $Y$.
\end{lemm}
\begin{proof} Recall that we denoted by $\tilde{\Delta}:=\text{det}\left(\tilde{\pi}_*\al O_{\tilde{X}_s }\right)^{-1}$, note that $\tilde{\Delta}$ is non-trivial $2-$torsion line bundle over $\tilde{Y}_s$.
By \cite{BL}, page 382, we know that there exists an even theta characteristic, say $\kappa''$, on $\tilde{Y}_s$ such that  $$h^0(\kappa'')\equiv h^0(\kappa''\otimes\tilde{\Delta})\equiv 0\mod 2,$$
if we set $\kappa=\tilde{\pi}^*\kappa''$, we get by the projection formula 
$$h^0(\kappa)=h^0(\kappa'')+h^0(\kappa''\otimes\tilde{\Delta})\equiv 0\mod 2,$$
hence $\kappa$ is even theta characteristic. Moreover, let $N$ be a line bundle on $\tilde{Y}_s$ such that $N^2\cong \al O(S)$ (recall that $S:=Ram(\tilde{Y}_s/Y)$ has an even degree), then by Hurwitz formula, we have $$(N^{-1}\otimes \kappa'')^2=(\tilde{q}^*\kappa')^2,$$
where $\kappa'$ is a (any) theta characteristic on $Y$. It follows that there exist a $2-$torsion line bundle $\alpha$ on $\tilde{Y}_s$ such that $$N^{-1}\otimes \kappa''\otimes\alpha=\tilde{q}^*\kappa'.$$
It suffices to take $M=\tilde{\pi}^*(N\otimes\alpha^{-1})$.
\end{proof}

\begin{rema} The determinant induces a morphism $\text{det}:\al U_X^{\sigma,+}(r)\ra P=\text{Nm}^{-1}(\al O_Y)$. Recall from the introduction that $P=\al U_X^{\sigma,+}(1)$.  The composition of this map with the direct image $q_*$ gives a map $$\al P^+\lra  P,$$ and each connected component of $\al P^+$ dominates a connected component of $P$. Indeed, let $L\in \al P^+$, and let $M$ be a line bundle in $\al P^+$ such that $\text{Nm}_{\tilde{X}_s/X}(M)=\delta$, where $\delta=(q_*\al O_{\tilde{X}_s})^{-1}$. Then $L\otimes M^{-1}$ is in the Prym variety of $\tilde{X_s}\ra \tilde{Y}_s$, hence can be written as $\tilde{\sigma}^*\lambda\otimes \lambda^{-1}$, it follows that $$\text{det}(q_*L)=\text{Nm}_{\tilde{X}_s/X}(L)\otimes \delta^{-1}=\sigma^*\text{Nm}_{\tilde{X}_s/X}(\lambda)\otimes \text{Nm}_{\tilde{X}_s/X}(\lambda)^{-1}.$$ Then it suffices to recall that the image of the map $\lambda\lra\tilde{\sigma}^*\lambda\otimes \lambda^{-1}$ equals the identity component of the Prym variety when $\lambda$ runs $\text{Pic}^0(\tilde{X})$ and equals the other component when it runs $\text{Pic}^1(\tilde{X})$.  \\ Moreover, this map $\al P^+\ra P$ is in fact surjective morphism. In particular, using Appendix A, this implies that for general line bundle $L$ in both connected components of $\al P^+$, $q_*L$ is stable.
\end{rema}

\subsubsection{Trivial determinant}
In this section, nothing is assumed on the cover ${\pi:X\ra Y}$, i.e. it may be ramified or not. 
Denote by $\al Q^+=\text{Nm}_{\tilde{X}_s /X}^{-1}(\delta)$, where $\delta=(\text{det}(q_*\al O_{\tilde{X}_s }))^{-1}$. For general $s\in W^{\sigma,+}$, $\al Q^+$ is  isomorphic to the Prym variety of the spectral cover $\tilde{X}_s\ra X$. In particular it is connected.

\begin{prop}\label{NCCI} For general $s\in W^{\sigma,+}$,  $\al P^+\cap\al Q^+$ is connected.
\end{prop}
\begin{proof}
Fixing an element in $\al P^+\cap \al Q^+$ gives by tensor product  an isomorphism $$\al P\cap\al Q\stackrel{\sim}{\longrightarrow}\al P^+\cap \al Q^+.$$ 
So it is sufficient to prove the connectedness of $\al P\cap \al Q$ (recall that $\al P$ and $\al Q$ are the Prym varieties of $\tilde{\pi}:\tilde{X}_s \ra \tilde{Y}_s$ and $q:\tilde{X}_s \ra X$ respectively).  The norm map $\widetilde{\text{Nm}}:J_{\tilde{X}_s }\ra J_{\tilde{Y}_s}$ induces a homomorphism $$\vartheta:\al Q\ra  \underline{\al Q},$$
which is just the restriction of $\widetilde{\text{Nm}}$ to $\al Q$, here $\underline{\al Q}$ is the Prym variety of $\tilde{q}:\tilde{Y}_s\ra Y$. We have a commutative diagram 
$$\xymatrix{ \al Q \;\ar@{^{(}->}[r] & J_{\tilde{X}_s }\ar[r]^\sim  & \widehat{J_{\tilde{X}_s }} \ar@{->>}[r]  & \widehat{\al Q}  \\ \underline{\al Q}\;\ar@{^{(}->}[r]\ar[u]^{\mu} & J_{\tilde{Y}_s} \ar[r]^\sim \ar[u]^{\tilde{\pi}^*} & \widehat{J_{\tilde{Y}_s}}  \ar[u]^{\widehat{\widetilde{\text{Nm}}}} \ar@{->>}[r] &   \widehat{\underline{\al Q}}\ar[u]_{\hat{\vartheta}}, }$$ 
where $\widehat{\al Q}$  and  $\widehat{\underline{\al Q}}$ are the dual abelian varieties of $\al Q$ and $\underline{\al Q}$ respectivaly and $\mu:\underline{\al Q}\ra \al Q$ is the morphism defined by the factorization  $$\tilde{\pi}^*|_{\underline{\al Q}}:\underline{\al Q}\stackrel{\mu}{\longrightarrow}\al Q\hookrightarrow J_{\tilde{X}_s }.$$ 
We obtain the commutative diagram
\begin{align} \label{diagram1}
\xymatrix{ \al Q  \ar[r]^{\varphi_{\al Q}}  & \widehat{\al Q}  \\ \underline{\al Q}\ar[u]^{\mu} \ar[r]^{\varphi_{\underline{\al Q}}}  &   \widehat{\underline{\al Q}}\ar[u]_{\hat{\vartheta}}, }
\end{align}
where $\varphi_{\al Q}:\al Q\ra \widehat{\al Q}$ (resp. $\varphi_{\underline{\al Q}}:\underline{\al Q}\ra \widehat{\underline{\al Q}}$)  is the restriction of the principal polarisation of $J_{\tilde{X}_s }$ (resp. $J_{\tilde{Y}_s}$) to $\al Q$ (resp. $\underline{\al Q}$). By \cite{BNR}, Remark $2.7$, as the spectral covers are always ramified, the types of these two polarisations are $(1,\dots,1,\underbrace{r,\dots,r}_{g_X})$ and $(1,\dots,1,\underbrace{r,\dots,r}_{g_Y})$ respectively, hence the degree of these two restrictions is $r^{2g_X}$ and $r^{2g_Y}$ respectively.\\ 

Assume that $X\ra Y$ is ramified. Then $\mu$ is injective. By Diagram (\ref{diagram1}) it follows that $$\text{card(Ker}(\varphi_{\al Q}\circ\mu))=\text{card(Ker}(\varphi_{\underline{\al Q}}\circ \hat{\vartheta})).$$
It is easy to see that $\text{card(Ker}(\varphi_{\al Q}\circ\mu))=r^{2g_Y}$. Indeed if $L\in\text{Ker}(\varphi_{\al Q}\circ\mu)$, then $\tilde{\pi}^*L=q^*M$ for some $M\in J_X[r]$, but this implies that $M$ descends to $Y$ (here we use Kempf's Lemma (\ref{kempf}) to prove that if $q^*M$ descends to $\tilde{Y}_s$ then $M$ descends to $Y$), hence $L=\tilde{q}^*N$ for some $N\in J_Y[r]$, i.e. $\text{Ker}(\varphi_{\al Q}\circ\mu)=\tilde{q}^*J_Y[r]$, as $\tilde{q}^*$ is injective (c.f.  \cite{BNR}, remark $3.10$) this implies the result.\\ 
But we also have $\text{Ker}(\varphi_{\underline{\al Q}})=\tilde{q}^*J_Y[r]$, so  $\text{card(Ker}(\varphi_{\underline{\al Q}}))=r^{2g_Y}$. This proves that $\hat{\vartheta}$ is injective. By general theory of abelian varieties (see for example \cite{BL}, Proposition $2.4.3$), $\text{Ker}(f)$ and $\text{Ker }(\hat{f})$ have the same number of connected components for any surjective morphism $f:A\ra B$ between abelian varieties. Since $\vartheta$ is clearly surjective, it follows that $\text{Ker}(\vartheta)$ is connected, and by definition, the kernel of $\vartheta$ is $\al P\cap \al Q$. \\

If $X\ra Y$ is \'etale, then so is $\tilde{\pi}:\tilde{X}_s\ra \tilde{Y}_s$. In this case $\mu$ has degree $2$. Let  $L\in\text{Ker}(\varphi_{\al Q}\circ\mu)$, then as above $\tilde{\pi}^*L=q^*M$, for some $M\in J_X[r]$, hence $M$ descends to $Y$,  say $M=\pi^*N$, as $M^r=\al O_X$, we get $N^r=\al O_Y$ or $N^r=\Delta$, recall that $\Delta$ is the $2-$torsion line bundle attached to $X\ra Y$. Denote the set of $r^{th}$ roots of $\Delta$ by $T$.  It follows that $L=\tilde{q}^*N$ or $L=\tilde{q}^*N\otimes \tilde{\Delta}$ for $N\in J_Y[r]\cup T$, where $\tilde{\Delta}\in J_{\tilde{Y}_s}[2]$ is the line bundle attached to $\tilde{X}_s\ra \tilde{Y}_s$. Note that $\tilde{\Delta}=\tilde{q}^*\Delta$. Since $L\in\underline{\al Q}$, so 
\begin{itemize}
\item if $r$ is even, then multiplication by $\Delta$ is an involution of $J_Y[r]$ and $T$,  so $L\in \tilde{q}^*J_Y[r]$ (because $\tilde{q}^*T\cap\underline{\al Q}=\emptyset$). This implies that  $\text{Ker}(\varphi_{\al Q}\circ\mu)=\tilde{q}^*J_Y[r]$, hence $\text{card(Ker}(\varphi_{\al Q}\circ\mu))=r^{2g_Y}$. It follows that $\hat{\vartheta}$ is injective. So  $\al P\cap \al Q$ is connected.
\item if $r$ is odd, then multiplication by $\Delta$ is an isomorphism $J_Y[r]\cong T$, and since $q^*\Delta=\tilde{\Delta}$ we can assume that $N\in J_Y[r]$. As $\text{Nm}_{\tilde{Y}_s/Y}(\tilde{q}^*N\otimes\tilde{\Delta})=\Delta^r=\Delta$, then $\tilde{q}^*N\otimes\tilde{\Delta}\not\in \underline{\al Q}$. It follows $L\in \tilde{q}^*J_Y[r]$, hence $\text{card(Ker}(\varphi_{\al Q}\circ\mu))=r^{2g_Y}$, and so we deduce again the connectedness of $\al P\cap \al Q$.
\end{itemize}
\end{proof}
\begin{theo}
The pushforward map $$q_*:\al P^+\cap\al Q^+\dashrightarrow \al {SU}_X^{\sigma,+}(r)$$ is dominant. In particular $\al {SU}_X^{\sigma,+}(r)$ is irreducible.
\end{theo}
\begin{proof}
Let $P_0$ be the identity component of the Prym variety of $X\ra Y$, then it is clear that the map $$\al {SU}_X^{\sigma,+}(r)\times P_0\lra\al {U}_{X,0}^{\sigma,+}(r)$$ is surjective, where $\al {U}_{X,0}^{\sigma,+}(r)$ is the connected component of $\al {U}_X^{\sigma,+}(r)$ which is over $P_0$ by the determinant map. It follows by Theorems \ref{main1} and \ref{etalecase} (in the ramified and \'etale cases respectively) that for general $E \in \al {U}_{X,0}^{\sigma,+}(r)$, there exists $L\in \al P^+$ such that $q_*L=E$. Let $\lambda\in P_0$ be an $r^{th}$ root of $\text{det}(E)^{-1}$. It follows by the projection formula that $q_*(L\otimes q^*\lambda)=E\otimes \lambda\in \al{SU}^{\sigma,+}_X(r)$. Note that $L\otimes q^*\lambda\in\al P^+$ because $\widetilde{Nm}(q^*\lambda)=\al O_{\tilde{Y}_s}$.  Hence a general $E\in\al{SU}^{\sigma,+}_X(r)$ can be written as a direct image of some $L\in\al P^+$. But since $$\text{det}(q_*L)=\delta^{-1}\otimes \text{Nm}_{\tilde{X}_s/X}(L),$$ where $\delta=\text{det}(q_*\al O_{\tilde{X}_s })^{-1}$, we deduce that if $q_*L$ has trivial determinant then  $\text{Nm}_{\tilde{X}_s/X}(L)=\delta$, thus $L\in \al Q^+$. So we get a dominant rational map $$\al P^+\cap\al Q^+\lra \al{SU}_X^{\sigma,+}(r).$$ Now by Proposition \ref{NCCI}, $\al P^+\cap\al Q^+$ is connected. This ends the proof.
\end{proof}
\begin{rema} Remark that the map $$\al{SU}_X^{\sigma,+}(r)\times P\lra\al U_X^{\sigma,+}(r)$$
is surjective, unless $\pi:X\ra Y$ is \'etale and $r$ is even, for which its image is one connected component. Indeed, the ramified case is clear, so assume that $\pi$ is \'etale, then if $r$ is odd, the map $[r]:P\ra P$ is surjective, and its image is the identity component $P_0\subset P$ when $r$ is even. Now use Theorem \ref{etalecase} to deduce the result. \\
\end{rema}
\subsection{$\sigma-$alternating case}

\subsubsection{The ramified case}\label{ramifiedalterne} Suppose that $\pi:X\ra Y$ is ramified and $r$ is even.  Let $E$ be a $\sigma-$alternating  stable vector bundle, consider  the involution $f$ on the space \mbox{$H^0(X,E\otimes \sigma^*E\otimes K_X)$} associated to the linearisation $\ak t\otimes \eta$ on $E\otimes \sigma^*E\otimes K_X$, where $\ak t$ is the linearisation on $E\otimes \sigma^*E$ equals the transposition. \\ 
 Let $\phi\in H^0(X,E\otimes \sigma^*E\otimes K_X)_-$. Using the isomorphism $\psi:\sigma^*E\cong E^*$, we can see $\phi$ as a map $E\ra E\otimes K_X$ (in fact we just identify $\phi$ and $\phi\circ (\,^t\psi)$ to simplify the notations).  Then we have
\begin{lemm} The following diagram $$\xymatrix{\sigma^*\phi:&\sigma^*E\ar[r]&\sigma^*E\otimes\sigma^* K_X\ar[d]^{\psi\otimes \eta}\\ ^t\phi:&E^*\ar[r] \ar[u]^{\psi^{-1}} &E^*\otimes K_X,}$$ commutes.
\end{lemm}
\begin{proof} Write (locally)  $$\phi=\sum_ks_k\otimes\sigma^*t_k\otimes\alpha_k,$$ and let $v$ be a local section of $E^*$, then we have 

\begin{align*} 
(\psi\otimes\eta)\circ (\sigma^*\phi)\circ \psi^{-1}(v)&=(\psi\otimes\eta)\left(\sum_k\gen{(\sigma^*\psi)(t_k),\psi^{-1}(v)}\sigma^*s_k\otimes\sigma^*\alpha_k\right)\\ 
&=(\psi\otimes\eta)\left(\sum_k-\gen{(\,^t\psi)(t_k),\psi^{-1}(v)}\sigma^*s_k\otimes\sigma^*\alpha_k\right)\\
&=-(\psi\otimes\eta)\left(\sum_k\gen{t_k,v}\sigma^*s_k\otimes\sigma^*\alpha_k\right) \\
&=-\sum_k\gen{t_k,v}\psi(\sigma^*s_k)\otimes\eta(\sigma^*\alpha_k)\\
&=-\left(\sum_kt_k\otimes\psi(\sigma^*s_k)\otimes\eta(\sigma^*\alpha_k)\right)(v)\\
&= - \,^tf(\phi)(v)\\
&=\,^t\phi(v).
\end{align*}
Thus 
$$^t\phi=(\psi\otimes\eta)\circ (\sigma^*\phi)\circ \psi^{-1}.$$
\end{proof}

In particular, over a ramification point $p\in R$,  we have 
\begin{equation} \label{phip}
^t\phi_p=\psi_p\cdot \phi_p\cdot \psi_p^{-1}.
\end{equation}  
\begin{lemm}\label{characteristicpolinomialisquare}
Let $J_r$ be the $r\times r$ matrix $$\begin{pmatrix}0&I_{r/2}\\-I_{r/2}& 0 \end{pmatrix},$$
where $I_{r/2}$ is the identity matrix of size $r/2$. Let $\al A$ be the set of matrices $A$ such that $^tA=J_rAJ_r^{-1}.$ Then the characteristic polynomial on $\al A$ that sends  $A$ to $\chi(A)$ is a square of a polynomial in the coefficients of $A$. In particular  $\text{det}(A)$  is a square too.
\end{lemm}
\begin{proof}
Let $B=A-\lambda I_r$, then $$^tB=\,^tA-\lambda I_r=J_r(A-\lambda I_r)J_r^{-1}=J_rBJ_r^{-1},$$ it follows that $^t(J_rB)=-J_rB,$ hence  $J_rB$ is anti-symmetric matrix,  thus $$\chi(A)=\text{det}(B)=\text{det}(J_rB)=\text{pf}(J_rB)^2,$$
where $\text{pf}(M)$ denote the Pfaffian of the anti-symmetric matrix $M$.
\end{proof}
For $s\in W^{\sigma,+}$ and $p\in X$, fix an isomorphism $(K_X)_p\cong \bb C$ and let 
\begin{equation*} P(x,p)= x^r+s_1(p)x^{r-1}+\cdots+s_r(p)\in\bb C[x].
\end{equation*}
Define  $$W^{\sigma,-}=\{s\in W^{\sigma,+}\,|\; P(x,p) \text{ is square for all } p\in R \}\subset W^{\sigma,+}.$$

\begin{prop}\label{wsigmaalternating}
The Hitchin morphism induces a map $$T^*\al U_X^{\sigma,-}(r)\lra W^{\sigma,-}.$$ Moreover, for each $s\in W^{\sigma,-}$ the associated spectral curve $\tilde{X}_s $ is singular and we have $\tilde{R}=q^{-1}(R)\subset \tilde{S}=Ram(\tilde{X}_s/X)$. And for general $s\in W^{\sigma,-}$, the singular locus of $\tilde{X}_s $ is exactly $\tilde{R}$ 
\end{prop}
\begin{proof}The first part follows directly from equation (\ref{phip}) and Lemma \ref{characteristicpolinomialisquare}.\\
using Lemma \ref{smoothnesslemma} we deduce that the fixed locus of an involution on smooth curve is smooth, hence reduced (if it is not empty). This implies that for any $s\in W^{\sigma,-}$, the associated spectral curve $\tilde{X}_s$ is singular at every point of $\tilde{R}$.
To see that these are the only singularities for general $s\in W^{\sigma,-}$, it is sufficient to show that the set of spectral data with such property is not empty in $W^{\sigma,-}$. For this, just take the spectral data  $s=(0,\cdots,0,\pi^*s_r)\in W^{\sigma,-}$ , where $s_r$ is a general section in $H^0(Y,K_Y^r\otimes \Delta^{r-1})$.  

 To see that $\tilde{R}\subset \tilde{S}$, recall that  by \cite{BNR}, Remark $3.3$, the discriminant of the polynomial $$x^r+q^*s_1x^{r-1}+\cdots+q^*s_r$$ gives the ramification divisor  $\tilde{S}=Ram(\tilde{X}_s /X)$. In other words, a point $a\in\tilde{X}_s $ (over $p\in X$) is in $\tilde{S}$ if and only if $a$ is a multiple root of $P(x,p)$. Hence  we deduce $\tilde{R}\subset \tilde{S}$. 
\end{proof}

It is clear that $W^{\sigma,-}$ is not a linear subspace of $W^{\sigma,+}$. So this system is not integrable in the sense of Hitchin \cite{NH}. \\

Moreover, for general $s\in W^{\sigma,-}$, over each $p\in R$, the polynomial $P(x,p)$ is a square of a polynomial with \emph{simple roots}. Thus the singularities are ordinary double points. The condition that the polynomial $P(x,p)$ is a square of a polynomial with simple roots is given by $r/2$ equations, hence it decreases the dimension of $W^{\sigma,+}$ by $r/2$, for each $p\in R$. More precisly, if $D=\bb C[x]_{\leqslant r}$ is the vector space of polynomials of degree at most $r$, and $\al S\subset D$ is the locus of square polynomials. Then $W^{\sigma,-}$ can be defined as the pullback of $\bigoplus_{p\in R}\al S$ via the map  $$W^{\sigma,+}\lra \bigoplus_{p\in R}D,$$ which sends $s\in W^{\sigma,+}$ to the polynomials $(P(x,p))_{p\in R}.$ Since this map is a surjective linear map and because $\text{codim}_D(\al S)=r/2$ we deduce
\begin{align*}
\text{dim}(W^{\sigma,-})&=\text{dim}(W^{\sigma,+})-2n\frac{r}{2}\\&=\dfrac{r^2}{2}(g_X-1)+\dfrac{nr}{2}-nr\\&=\dfrac{r^2}{2}(g_X-1)-\dfrac{nr}{2}\\ &=\text{dim}(\al U_X^{\sigma,-}(r)).
\end{align*}

Let $s\in W^{\sigma,-}$ be general such that the singular locus of  $\tilde{X}_s $ is $\tilde{R}$ and all singularities are nodes. Denote by $\ak q:\hat{X}_s\ra \tilde{X}_s $ its normalisation, then the genus  $g_{\hat{X}_s}$ of $\hat{X}_s$ is given by
 \begin{align} \label{genusnormalization}
 g_{\hat{X}_s}&= (\text{arithmetic genus of } \tilde{X}_s ) -  (\text{number of singular points}) \nonumber \\ 
            &= r^2(g_X-1)+1- \frac{1}{2} \text{deg}(\tilde{R}) \\
            &= r^2(g_X-1)+1 - rn. \nonumber
 \end{align}
 
 \begin{lemm}
 	Let $\hat{\sigma}$ the lifting of the involution $\tilde{\sigma}$ to $\hat{X}_s$. Then $\hat{\sigma}$ has no fixed points ($\hat{\sigma}$ interchanges the two points over each singular point). Moreover, we have $$\hat{X}_s/\hat{\sigma}\cong \tilde{X}_s/\tilde{\sigma}=\tilde{Y}_s.$$ 
 \end{lemm}
\begin{proof} If $t$ is local parameter near $p\in R$ and $x$ is a local parameter induced by the tautological section $x$ of the pull back of $K_X$ to $|K_X|$  in a neighbourhood of a ramification point $\lambda\in \tilde{R}$ over $p$, then by definition,  $\tilde{\sigma}$  sends $t\ra-t$ and $x\ra x$, and we can write the equation of $\tilde{X}_s $ near $\lambda$ as $$x^2+t^2+(\text{higher terms}).$$  Then it is clear that $\tilde{\sigma}$ interchanges the two tangent lines at this singular point.  Thus it interchanges the two branches of $\tilde{X}_s$ over  $\lambda$.
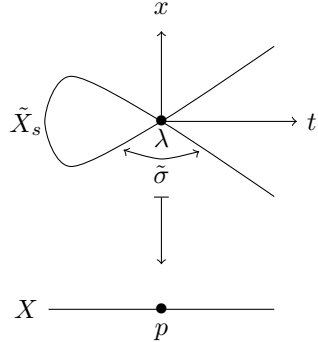
$$ \begin{tikzpicture} 
\draw [black] plot [smooth] coordinates {(1.5,1) (0,0) (-1.2,-0.6) (-1.55,0) (-1.2,0.6) (0,0) (1.5,-1)};
\foreach \Point in {(0,0)}{ \node at \Point {\textbullet};}
\draw[|->] (0,-1) -- (0,-1.9);
\draw[->] (0,0) -- (1.8,0);
\draw[->] (0,0) -- (0,1.2);
\draw[->] (-.48,-.39) -- (-.5,-.38);
\draw[->] (.48,-.41) -- (.5,-.4);
\draw [black] plot [smooth] coordinates {(-1.5,-2.5) (1.5,-2.5)};
\draw [-] plot [smooth] coordinates {(-.5,-0.38) (0,-.5) (0.5,-0.4)};
\foreach \Point in {(0,-2.5)}{ \node at \Point {\textbullet};}
\node (A) at (0, -0.7) {$\tilde{\sigma}$};
\node (B) at (-1.8, -2.5) {$X$};
\node (C) at (-1.8, 0) {$\tilde{X}_s $};
\node (D) at (0, -.22) {$\lambda$};
\node (p) at (0, -2.8) {$p$};
\node (x) at (0, 1.5) {$x$};
\node (t) at (2, 0) {$t$};
\end{tikzpicture}  $$

Now $\ak{q}$ induces a map $\hat{X}_s/\hat{\sigma}\lra \tilde{X}_s/\tilde{\sigma}$ which is an isomorphism outside the branch points of $\tilde{X}_s\ra \tilde{Y}_s$. But we see also that it is a one-to-one also over this locus. Since $\tilde{X}_s/\tilde{\sigma}$ is smooth (this can be seen locally using the equation of $\tilde{X}_s$), we deduce that this bijection is an isomorphism. 
\end{proof}
Let $\hat{\pi}:\hat{X}_s\ra \tilde{Y}_s$, and $\al P=\text{Prym}(\hat{X}_s/\tilde{Y}_s)$, then we have
\begin{align*}
\text{dim}(\al P)&= g_{\tilde{Y}_s}-1   \;\;\;\;\;\;\;\;(\hat{\pi} \text{ is \'etale})\\
                 &= \dfrac{1}{2}(g_{\hat{X}_s}-1) \;\;\;\;\;\; \text{ (Riemann-Roch) }\\ 
                 &= \dfrac{r^2}{2}(g_{X}-1)-\dfrac{rn}{2} \;\;\;\;\;\text{(by formula (\ref{genusnormalization}))}  \\
                 &= \text{dim}(\al U_X^{\sigma,-}).
\end{align*}

Recall that we denoted by $\tilde{q}:\tilde{Y}_s\ra Y$, $S=Ram(\tilde{Y}_s/Y)$. Let $\hat{\Delta}=\text{det}(\hat{\pi}_*\al O_{\hat{X}_s})^{-1}$ and $\hat{S}=Ram(\hat{X}_s/X)$. The line bundle $\al O(\hat{S})$ descends to $\tilde{Y}_s$, and we have $$\al O(\hat{S})=\ak q^*\al O(\tilde{S}-\tilde{R})=\hat{\pi}^*(\al O(S)\otimes \tilde{q}^*\Delta^{-1}),$$ this induces a linearisation on $\al O(\hat{S})$, which we call \emph{positive}, and we fix it hereafter.\\
 The line bundles $L$ on $\hat{X}_s$ such that $$\hat{\sigma}^*L\cong L^{-1}(\hat{S}),$$ with a $\hat{\sigma}-$alternating isomorphism (see Lemma \ref{normlinebundle}) are those with norm (with respect to $\hat{\pi}$) equals $\al O(S)\otimes({q'}^*\Delta^{-1}) \otimes \hat{\Delta}$. We denote this subvariety of line bundles by $\hat{\al P}$. 
Denote by $\hat{q}$ the map $\hat{X}_s\ra X$. We have 
\begin{theo}\label{main2}
Suppose that $X\ra Y$ is ramified. For general $s\in W^{\sigma,-}$, the pushforward map $$\hat{q}_*:\hat{\al P}\dashrightarrow \al U_X^{\sigma,-}(r)$$
is dominant. In particular $\al U_X^{\sigma,-}(r)$ has two irreducible components.
\end{theo}
\begin{proof}
As in Theorem \ref{main1} we deduce that this map is well-defined. Moreover, using Theorem \ref{denseverystable}, we deduce that the map: $$\Pi:T^*\al U_X^{\sigma,-}(r)\ra \al U_X^{\sigma,-}(r)\times W^{\sigma,-}$$ is dominant. Hence, for general $s\in W^{\sigma,-}$, we get a dominant map $$\sr H^{-1}(s)\lra \al U_X^{\sigma,-}(r).$$ We claim that $\sr H^{-1}(s)$ is a non-empty open set of $\hat{\al P}$. By Proposition \ref{bnrprop} $\sr H^{-1}(s)$ is in bijection with an open set of isomorphism classes of rank one torsion-free $\al O_{\tilde{X}_s}-$modules. Given such  a torsion-free $\al O_{\tilde{X}_s}-$module  $\sr F$, we have $\tilde{\sigma}^*\sr F\cong \sr F^*(\tilde{S})$ (follows from the exact sequence (\ref{seq1})). For general $s\in W^{\sigma,-}$, the divisor $\frac{1}{2}\tilde{R}$ is reduced, consider the line bundle $L=\ak q^*(\sr F(-\frac{1}{2}\tilde{R}))$ on $\hat{X}_s$, it  verifies 
\begin{align*}
\hat{\sigma}^*L&\cong \ak q^* (\tilde{\sigma}^*\sr F(-\dfrac{1}{2}\tilde{R})) \\ 
               &\cong \ak q^*(\sr F^*(-\dfrac{1}{2}\tilde{R}+\tilde{S})) \\
               &\cong (\ak q^*\sr F(-\dfrac{1}{2}\tilde{R}))^{-1}(\hat{S}) \\
               &\cong L^{-1}(\hat{S}).
\end{align*}
In fact the isomorphism $\hat{\sigma}^*L\cong L^{-1}(\hat{S})$ is induced by $\psi$, hence it is $\hat{\sigma}-$alternating, thus $L\in\hat{\al P}$. \\ Conversely, given $L\in \hat{\al P}$ such that $\hat{q}_*L$ is stable, then by duality of finite flat morphisms we deduce that $\hat{q}_*L$ is $\sigma-$alternating anti-invariant vector bundle.

We have seen that the involution $\hat{\sigma}$ has no fixed point, hence $\hat{\al P}$ has two connected components distinguished by the parity of $$h^0(L\otimes \hat{q}^*\kappa)=h^0(\hat{q}_*L\otimes \kappa),$$ where $\kappa$ is a theta characteristic over $X$. It follows that the image of the two connected components of $\hat{\al P}$ can't intersect. Moreover, for some $\sigma-$invariant square root $\alpha$ of $\al O(R)$, we have $\kappa=\alpha\otimes\pi^*\kappa'$, where $\kappa'$ is a theta characteristic over $Y$. Since the $\sigma-$bilinear form $\tilde{\psi}:E\otimes\sigma^*E\lra \al O_X$ is $\sigma-$alternating, that's to say $$\tilde{\psi}(s\otimes\sigma^*t)=-\nu(\sigma^*(\tilde{\psi}(t\otimes\sigma^*s))),$$ where $\nu:\sigma^*\al O_X\ra \al O_X$ is the positive linearisation. Taking the tensor product with $\alpha$ we get a bilinear form  $$(E\otimes\alpha)\otimes\sigma^*(E\otimes\alpha)\lra \al O_X(R),$$ which induces a symmetric non-degenerate bilinear form $$\pi_*(E\otimes\alpha)\otimes\pi_*(E\otimes\alpha)\lra\al O_Y,$$ hence, using the result of Mumford \cite{M2}, the map $\al U_X^{\sigma,-}(r)\ra \bb Z/2$ given by $$E\lra h^0(\pi_*(E\otimes\alpha)\otimes \kappa')\mod 2$$ is constant under deformation of $E$. This implies that $\al U_X^{\sigma,-}(r)$ has two connected components.
\end{proof}

\subsubsection{The \'etale case} Assume now that $\pi:X\ra Y$ is \'etale. 
\begin{proof}[Proof of Theorem \ref{main1.2}] It is clear that $W^{\sigma,-}=W^{\sigma,+}$, and for general $s\in W^{\sigma,-}$, the associated spectral curve $\tilde{X}_s$ is smooth and the attached involution $\tilde{\sigma}$ has no fixed points. Define $$\al P^-=\widetilde{\text{Nm}}^{-1}(\al O(S)\otimes \tilde{\Delta})\subset \text{Pic}^m(\tilde{X}_s).$$
Since we have the positive linearisation on $\al O(\tilde{S})$, it follows that the isomorphism $\tilde{\sigma}^*L\stackrel{\sim}{\lra} L^{-1}(\tilde{S})$ is $\tilde{\sigma}-$alternating. Moreover, let $\xi$ be a line bundle over $X$ of norm $\Delta$, and $E\in \al U_X^{\sigma,+}(r)$. As $\sigma^*\xi\ra \xi^{-1}$ is $\sigma-$alternating, then the isomorphism  $\sigma^*(E\otimes\xi)\stackrel{\sim}{\lra} (E\otimes \xi)^*$ is $\sigma-$alternating too. Since  $q^*\Delta=\tilde{\Delta}$, we have $$\widetilde{\text{Nm}}(q^*\xi)=q^*\text{Nm}(\xi)=\tilde{\Delta},$$ it follows that $q^*\xi$ induces by tensor product an isomorphism $\al P^+\cong \al P^-$. This with Theorem \ref{etalecase} end the proof of Theorem \ref{main1.2}. 
\end{proof}

\begin{rema} \label{det.alt} In this case the determinant induces a morphism $\text{det}:\al U_X^{\sigma,-}(r)\ra P=\text{Nm}^{-1}(\al O_Y)$ if $r$ is even, and $\text{det}:\al U_X^{\sigma,-}(r)\ra P':=\text{Nm}^{-1}(\Delta)$ if $r$ is odd.
\end{rema}

\section{The Hitchin system for invariant vector bundles}
We have seen in Remark \ref{parabolic} that $\sigma-$invariant vector bundles of fixed type $\tau$ correspond to parabolic vector bundles over $Y$ with parabolic structures associated to $\tau$ at the ramification points. The Hitchin systems for parabolic vector bundles have been studied in the \emph{smooth} case by Logares and Martens \cite{LM}. In this special case we have an explicit description of the fibers of the Hitchin map depending on the considered type, as well as a dominance result as in the case of anti-invariant vector bundles. We treat also the singular case.
 
We use results and notations of the previous section. We always suppose  that the cover $X\ra Y$ is ramified, the \'etale case is trivial. Fix the positive linearisation on $\al O_X$ and the \emph{negative} linearisation on $K_X$. Denote by $$\rho:\sigma^*K_X\ra K_X$$ this linearisation. We have $\bar{\sigma}(x)=-x$, where   $\bar{\sigma}$ is the involution on $\bb S=\bb P(\al O_X\oplus K_X^{-1})$ induced by these linearisations and $x$ is the tautological section of the pullback of $K_X$ to $\bb S$.

 Let  $$W^{\sigma,\ak m}=\bigoplus_{i=1}^rH^0(K_X^i)_+.$$ The $\ak m$ in the notation refers to \emph{maximal} types.\\
For simplicity, we assume hereafter that the degree $d$ of the $\sigma-$invariant vector bundles is $0$.

\begin{lemm} \label{5.1}
Let  $s\in W^{\sigma,\ak m}$, $\tilde{X}_s $ the associated spectral curve, then $\sigma$ lifts to an involution $\tilde{\sigma}$ on $\tilde{X}_s $. Moreover, for general such $s$, we have
\begin{itemize}
\item If $r$ is even, then this involution has no fixed point.
\item If $r$ is odd, this involution has just $2n$ fixed points.
\end{itemize}
\end{lemm}
\begin{proof}
Consider $s=(s_1,\cdots,s_r)\in W^{\sigma,\ak m}$, so we have $s_i(\sigma(p))=(-1)^is_i(p)$ for each point $p\in X$. Let $p\in X$, and $x_0=[x_0:1]\in (\tilde{X}_s)_p$, then the involution $\bar{\sigma}$ on $\bb P(\al O_X\oplus K_X^{-1})$ attached to the fixed linearasations on $\al O_X$ and $K_X^{-1}$ sends $x_0$ to $y_0=[-x_0:1]$ in $(\tilde{X}_s)_{\sigma(p)}$, but $y_0$ is a solution to $$x^r+s_1(\sigma(p))x^{r-1}+\cdots+s_r(\sigma(p))=0.$$ 
Indeed  
\begin{align*}
y_0^r+s_1(\sigma(p))y_0^{r-1}+\cdots+s_r(\sigma(p)) &= (-x_0)^r+(-s_1(p))(-x_0)^{r-1}+\cdots+(-1)^rs_r(p)\\&= (-1)^r\left(x_0^r+s_1(p)x_0^{r-1}+\cdots+s_r(p)\right) \\&=0.
\end{align*}
hence $\bar{\sigma}(x_0)$ is in $(\tilde{X}_s)_{\sigma(p)}$, thus $\bar{\sigma}$ induces an involution on $\tilde{X}_s $ which we are looking for.\\
Note that $0$ is the only fixed point of $\tilde{\sigma}$ over a ramification point. One remarks that for odd $i$,  $s_i(p)=0$, for any $p\in R$. Suppose that $r$ is odd, this implies that $0$ is always in $(\tilde{X}_s)_p$, and for general $s\in W^{\sigma,\ak m}$, it is a simple root of the equation above, hence there are just $2n$ fixed points in $\tilde{X}_s $. If $r$ is even, we deduce that for general $s$, $\tilde{\sigma}$ has no fixed point.  
\end{proof} 
 Let as before $\tilde{\pi}:\tilde{X}_s\ra\tilde{Y}_s:=\tilde{X}_s/\tilde{\sigma}$. Using Riemann-Roch formula, we get for general $s\in W^{\sigma,\ak m}$: $$g_{\tilde{Y}_s}=\dfrac{1}{2}g_{\tilde{X}_s }+\dfrac{1}{2}-\dfrac{k}{2},$$ where $k$ is the half of the number of fixed points of $\tilde{\sigma}$. \\ Let $\text{Pic}^m(\tilde{X}_s)^{\tilde{\sigma}}$ the locus of $\tilde{\sigma}-$invariant line bundles of degree $m$ over $\tilde{X}_s$. Since $g_{\tilde{X}_s }=r^2(g_X-1)+1$ and $g_X-1=2(g_Y-1)+n$, we deduce from Lemma \ref{5.1}
\begin{align*}\text{dim}(\text{Pic}^{m}(\tilde{X}_s )^{\tilde{\sigma}})=g_{\tilde{Y}_s}&=\begin{cases} 
r^2(g_Y-1)+n\frac{r^2}{2}+1 &r\equiv 0\mod2 \\
r^2(g_Y-1)+n\frac{r^2-1}{2}+1 &r\equiv 1\mod2 
\end{cases}\\&=\text{dim}(\al U_X^\sigma(r,0)),
\end{align*}
where the last equality is due to Proposition \ref{invariantdimension}.
 
\begin{rema}For general $s\in W^{\sigma,\ak m}$, $\tilde{X}_s $ is smooth. Indeed, taking $s=(0,\cdots,0,s_r)\in W^{\sigma,\ak m}$, where $s_r\in H^0(K_X^r)_+$ is a general global section which vanishes at most with multiplicity one at every ramification point. Then, by the proof of Lemma \ref{smoothlemma}, we deduce that $\tilde{X}_s$ is smooth. Since $W^{\sigma,\ak m}$ is irreducible, it follows that the set of $s\in W^{\sigma,\ak m}$ is dense. 
\end{rema}

Let $E$ be a stable $\sigma-$invariant vector bundle.  Recall from subsection \ref{infinvariant} that we have considered the involutions on $H^0(X, E\otimes E^*\otimes K_X)$ and $H^1(X, E\otimes E^*)$ induced by the canonical isomorphism $\sigma^*(E\otimes E^*)\ra E\otimes E^*$ (which is independent of the choice of a linearisation on $E$) and the linearisation $\rho$ on $K_X$.  By Lemma \ref{serre}, Serre duality is equivariant with respect to these involutions, i.e. $$H^1(X,E\otimes E^*)_+^*\stackrel{\sim}{\longrightarrow} H^0(X, E\otimes E^*\otimes K_X)_+.$$
Further we have
\begin{prop}
The Hitchin morphism induces a map 
$$\sr H_E:H^0(X, E\otimes E^*\otimes K_X)_+\longrightarrow  W^{\sigma,\ak m}=\bigoplus_{i=1}^r H^0(X,K_X^i)_{+}.$$
Moreover, we have an equality of dimensions $$\text{dim}(\al U_X^\sigma(r,0))=\text{dim}(W^{\sigma,\ak m}).$$
\end{prop}
\begin{proof}
By the proof of Proposition \ref{Hitchinmapequivarience}, we deduce that $$\sr H_i(f(\phi))=\rho^{\otimes i}(\sigma^*(\sr H_i(\phi))), $$ where $f$ is the involution on $H^0(X, E\otimes E^*\otimes K_X)$. Here, one should make a similar explicit local description of $\sr H_i$, this implies the first part of the lemma.\\
As we use the negative linearisation on $K_X$, by Remark \ref{propersubspaces}, it follows 
$$\begin{cases} h^0(X,K_X^i)_+=h^0(Y,K_Y^i\otimes\Delta^i)=(2i-1)(g_Y-1)+in & i\equiv 0\mod 2\\
                h^0(X,K_X^i)_+=h^0(Y,K_Y^i\otimes\Delta^{i-1})=(2i-1)(g_Y-1)+(i-1)n & i\equiv 1\mod 2,\; i\geqslant 3 \\
                h^0(X,K_X)_+ =h^0(Y,K_Y)=g_Y \end{cases}$$
taking the sum, we get $$\sum_{i=1}^rh^0(X,K_X^i)_{+}=\begin{cases} 
r^2(g_Y-1)+n\frac{r^2}{2}+1 &r\equiv 0\mod2 \\
r^2(g_Y-1)+n\frac{r^2-1}{2}+1 &r\equiv 1\mod2 
\end{cases}.$$\\
\end{proof}

\subsection{Smooth case}
Recall from Remark \ref{maxtype} that we have defined a maximal type to be a type $\tau$ such that $\al U_X^{\sigma,\tau}(r,0)$ has maximal dimension, and we have denoted the set of such types by $\ak{MAX}$. 
\begin{theo}\label{main3}Let $s\in W^{\sigma,\ak m}$ such that $\tilde{X}_s $ is smooth. Then the direct image map induces a dominant map $$q_*:\text{Pic}^m(\tilde{X}_s )^{\tilde{\sigma}}\dashrightarrow \al U_X^{\sigma,\ak{m}}(r,0),$$
where $\al U_X^{\sigma,\ak{m}}(r,0)$ is the moduli space of $\sigma-$invariant vector bundles of type $\tau\in\ak{MAX}$ (see Remark \ref{maxtype}). Moreover, for each type $\tau \in \ak{MAX}$, there exists a unique types $\tilde{\tau}$ of invariant line bundles in $\text{Pic}(\tilde{X}_s )$, such that we have a dominant map $$q_*:\text{Pic}^m(\tilde{X}_s )^{\tilde{\sigma},\tilde{\tau}}\dashrightarrow \al U_X^{\sigma,\tau}(r,0).$$ 
\end{theo}
\begin{proof}
It is clear that $L$ is $\tilde{\sigma}-$invariant if and only if $q_*L$ is $\sigma-$invariant. \\
By the proof of Theorem \ref{main1}, for general $\sigma-$invariant vector bundle $E$ of type $\tau\in\ak{MAX}$, the restriction of the Hitchin map $$\sr H_E:H^0(X,E\otimes E^*\otimes K_X)_+\ra W^{\sigma,\ak m} $$ is dominant. This implies that the map $$\Pi:T^*\al U_X^{\sigma,\ak{m}}(r,0)\lra \al U_X^{\sigma,\ak{m}}(r,0)\times W^{\sigma,\ak m}$$ is dominant too. Moreover, if we fix a type $\tau\in\ak{MAX}$, then there exists a corresponding type $\tilde{\tau}$ of $\tilde{\sigma}-$invariant line bundles on $\tilde{X}_s $, such that $$\text{Pic}^m(\tilde{X}_s )^{\tilde{\sigma},\tilde{\tau}}\dashrightarrow \al U_X^{\sigma,\tau}(r,0)$$ is dominant.
The type $\tilde{\tau}$ is constructed as follows: Suppose that $r$ is odd (the even case is trivial since $\tilde{\sigma}$ has no fixed point by Lemma \ref{5.1}). Remark first that $\tilde{\tau}\in\{+1,-1\}^{2n}/\pm$.   if $p\in R$ is such that $k_p=(r+1)/2$, then over such $p$, take $-1$ in  $\tilde{\tau}$, and $+1$ over the rest of points in $R$ (note that, because $s\in W^{\sigma,\ak m}$ is general, over each $p\in R$ there is just one fixed point by $\tilde{\sigma}$, so we identify $\tilde{R}$ with $R$ in this case). If $L$ is $\tilde{\sigma}-$invariant line bundle over $\tilde{X}_s$ of type $\tilde{\tau}$, then using the identification  $(q_*L)_p\cong \bigoplus_{x\in q^{-1}(p)}L_x$, we see easly that the type of $q_*L$ is $\tau$.\\  
\end{proof}
 
We prove in the next section that the only types corresponding to smooth spectral covers of $X$ are the maximal ones.

\subsection{General case}\label{generalcase}
Now let $\tau$ be any type, for simplicity of notations we suppose that we have just one point $p\in R$ such that $k_p< [\dfrac{r}{2}]$, where $[\;]$ stands for the floor function. In fact we can always suppose  $k_p\leqslant [r/2]$ due to taking the tensor product by $\al O_{X}(p)$. And for all other ramification point $a\not=p$ , $k_a$ is maximal (that's $k_a=r/2$ if $r$ is even, and $k_a=(r\pm1)/2$ if $r$ is odd).\\
To get a vector bundle of such type as a direct image, we should have $r-2k_p$ fixed points by $\tilde{\sigma}$ above $p$. Indeed, let $x$ be a point of $\tilde{X}_s$ above $p$, as  $\tilde{\sigma}$ interchanges the two fibers $L_x$ and $L_{\tilde{\sigma}(x)}$, its matrix over these two points is given by 
$$\begin{pmatrix} 0 & 1 \\ 1 & 0 \end{pmatrix} \sim \begin{pmatrix} 1 & 0 \\ 0 & -1 \end{pmatrix}.$$   Hence, one should have $2k_p$ non-fixed points to get the $-1$ eigenvalue with multiplicity $k_p$. This intuition is proved in the following theorem 
\begin{theo}
Let $\tau$ be a type as above, and denote by $$W^{\sigma,\tau}=H^0(K_X)_+\oplus\cdots \oplus H^0(K_X^{2k_p+1})_+\oplus \bigoplus_{i=2k_p+2}^{r}H^0(X,K_X^{i}(-(i-2k_p-1)p))_+.$$
Then, for any $\sigma-$invariant stable vector bundle $E$ of type $\tau$, the Hitchin map factorises through $W^{\sigma,\tau}$ giving a map $$\sr H_E:H^0(E\otimes E^*\otimes K_X)_+\lra W^{\sigma,\tau}.$$
Moreover, $$\text{dim}(\al U_X^{\sigma,\tau}(r,0))=\text{dim}(W^{\sigma,\tau}).$$ 
\end{theo}    
\begin{proof}
First we verify the dimensions. By Proposition \ref{invariantdimension}, we have $$\text{dim}(\al U_X^{\sigma,\tau}(r,0))=r^2(g_Y-1)+1+k_p(r-k_p)+\begin{cases} (2n-1)\frac{r^2}{4} &\;\;\text{ if }r\equiv 0\mod2 \\ (2n-1)\frac{r^2-1}{4}&\;\;\text{ if }r\equiv 1\mod2\end{cases}.$$
Recall that we have fixed the negative linearisation on $K_X$. By Lefschetz fixed point formula we deduce that 
\begin{equation} \label{cases} h^0(X,K_X^k(-ip))_+=(2k-1)(g_Y-1)+\begin{cases} kn-\frac{i}{2} &  k\equiv 0,\, i\equiv 0\mod2 \\ 
 kn-\frac{i+1}{2} & k\equiv 0,\, i\equiv 1\mod2 \\
 (k-1)n-\frac{i}{2} & k\equiv 1,\,  i\equiv 0\mod2 \\
 (k-1)n-\frac{i-1}{2} & k\equiv 1,\, i\equiv 1\mod2 
\end{cases}.
\end{equation}
So the dimension of $W^{\sigma,\tau}$ is given by $$\text{dim}(W^{\sigma,\tau})= \text{dim}(W^{\sigma,\ak m})-\sum_{i=1}^{r-2k_p-1}d(i),$$
where $$d(i)=\begin{cases} \frac{i}{2} & \;\;\text{ if } i\equiv 0\mod 2\\ 
\frac{i+1}{2} & \;\;\text{ if } i\equiv 1\mod 2 \end{cases}.$$
  
By a simple computation, we get $$\sum_{i=1}^{r-2k_p-1}d(i)=\begin{cases} \frac{r^2}{4}-k_p(r-k_p) & \;\;\text{ if } i\equiv 0\mod 2\\ 
\frac{r^2-1}{4}-k_p(r-k_p) & \;\;\text{ if } i\equiv 1\mod 2 \end{cases}.$$
It follows  that 
\begin{align*}
\text{dim}(W^{\sigma,\tau})&=\begin{cases} r^2(g_Y-1)+1+(2n-1)\dfrac{r^2}{4}+k_p(r-k_p) &\;\text{ if } r\equiv 0\mod2 \\ r^2(g_Y-1)+1+(2n-1)\dfrac{r^2-1}{4}+k_p(r-k_p) &\;\text{ if } r\equiv 1\mod2 \end{cases} \\  
&=\text{dim}(\al U_X^{\sigma,\tau}(r,0)).\\
\end{align*}

Now take $\phi\in H^0(X,E\otimes E^*\otimes K_X)_+$, let $\varphi:\sigma^*E\cong E$,  then the following diagram commutes $$\xymatrix{\phi:&E\ar[r]&E\otimes K_X\ar[d]^{\sigma^*(\varphi\otimes \rho)}\\ \sigma^*\phi:&\sigma^*E\ar[r] \ar[u]^{\varphi} &\sigma^*E\otimes \sigma^*K_X,}$$
 that's $$\sigma^*\phi=\sigma^*(\varphi\otimes \rho)\circ\phi\circ\varphi.$$
 In particular over $p$, one has
$$\phi_p=-A_p\phi_p A_p,$$
where $A_p=\text{diag}(\underbrace{-1,\cdots,-1}_{k_p},1,\cdots,1)$ is $r\times r$ diagonal matrix. 
 
This implies that  the matrix $\phi_p$  is of the form 
$$\begin{pmatrix}
0 & M \\ N & 0
\end{pmatrix},$$
where $M$ and $N$ are two matrices of type $k_p\times (r-k_p)$ and $(r-k_p)\times k_p$.\\

Now let $t$ be a local parameter on the neighbourhood of $p$, and denote $\phi(t)=\phi_p+t\phi'$ the restriction of $\phi$ to this neighbourhood. Then we get $$s_r(t)=\text{det}(\phi(t))=\text{det}(\phi_p+t\phi').$$ But   
\begin{equation}\label{matrix2}\phi_p+t\phi'=\begin{pmatrix} t\times * & M' \\ N' & t\times * \end{pmatrix}.\end{equation}
where $M'$ and $N'$ are two matrices of type $k_p\times (r-k_p)$ and $(r-k_p)\times k_p$ respectively, which are not necessarily divisible by $t$. Using the developpement of the determinant as a sum of monomials in the entries of the matrix, we see that every monomial contains at least $r-2k_p$ entries that belong to neither $M'$ nor $N'$, i.e. divisible by $t$. So we deduce that $\text{det}(\phi(t))$ is divisible by $t^{r-2k_p}$,  hence $$\text{det}(\phi(t))\in H^0(X,K_X^r(-(r-2k_p)p))_+.$$
But we have $$H^0(X,K_X^r(-(r-2k_p-1)p))_+=H^0(X,K_X^r(-(r-2k_p)p))_+,$$ because the first space is included in the second and they have the same dimension by formula (\ref{cases}).\\

The general case is treated similarly. Let $i\geqslant 2k_p+2$. Consider $\phi(t)$ as an element of $\al Mat_r(\bb C[[t]])$, and denote it by $\phi(t)=(a_{i,j}(t))_{i,j}$. As $\sr H_i(\phi)$ is by definition $(-1)^i\text{Tr}(\Lambda^i\phi(t))$. We just need to calculate the diagonal elements of the matrix $\Lambda^i\phi(t)$. Assume that $$\Lambda^i\phi(t)=(\alpha_{\underline{k},\underline{l}}(t)),$$ where $\underline{k}$ and $\underline{l}$ are $i$-tuples of strictly increasing integer in $\{1,\cdots,r\}$. Then if $\underline{k}=(k_1<\cdots<k_i)$ we have  $$\alpha_{\underline{k},\underline{k}}(t)=\text{det}(a_{k_l,k_{l'}}(t))_{1\leqslant l,l'\leqslant i}.$$ 
Hence, from the form of $\phi(t)$ given in (\ref{matrix2}), we deduce that $\sr H_i(\phi(t))$ is divisible by at least $t^{i-2k_p}$, hence $$\sr H_i(\phi)\in H^0(X,K_X^i(-(i-2k_p)p))_+=H^0(X,K_X^i(-(i-2k_p-1)p))_+.$$
Thus $$\sr H_E(\phi)\in W^{\sigma,\tau}.$$
\end{proof} 
Unfortunately, most of the types correspond to singular spectral curve $\tilde{X}_s \ra X$. But for general $s\in W^{\sigma,\tau}$ where $\tau$ is as above, the corresponding $\tilde{X}_s $ has just the point $0$ over $p$ which is singular with multiplicity $r-2k_p$. Moreover, this singularity is ordinary (the tangents at this point are distinct), we can see that using the equation defining $\tilde{X}_s$ and the generality of $s\in W^{\sigma,\tau}$. Hence the geometric genus  of the normalisation $\hat{X}_s$ of $\tilde{X}_s $ is equals to $$g_{\hat{X}_s}=r^2(g_X-1)+1-\dfrac{(r-2k_p)(r-2k_p-1)}{2}.$$
Moreover, the involution $\tilde{\sigma}$ lifts to an involution $\hat{\sigma}$ on $\hat{X}_s$ with $r-2k_p$ fixed points if $r$ is even and $r-2k_p+2n-1$ if $r$ is odd (Recall that we assumed for simplicity that we have just one point $p\in R$ with $k_p< [r/2]$). Indeed, if $t$ is a local parameter in a local neighbourhood of $p$ and $x$ is a local parameter near the ramification point $0\in \tilde{R}$ over $p$, then by definition,  $\tilde{\sigma}$  send $t\ra-t$ and $x\ra - x$, thus it does not interchange the two branches near $\lambda$ 
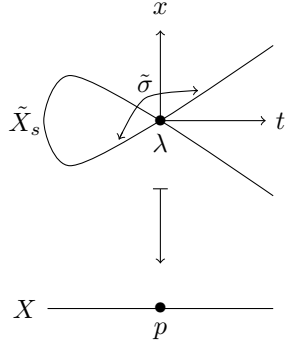
$$\begin{tikzpicture} 
\draw [black] plot [smooth] coordinates {(1.5,1) (0,0) (-1.2,-0.6) (-1.55,0) (-1.2,0.6) (0,0) (1.5,-1)};
\foreach \Point in {(0,0)}{ \node at \Point {\textbullet};}
\draw[|->] (0,-0.9) -- (0,-1.9);
\draw[->] (0,0) -- (1.4,0);
\draw[->] (0,0) -- (0,1.2);
\draw [black] plot [smooth] coordinates {(-1.5,-2.5) (1.5,-2.5)};
\draw [] plot [smooth] coordinates {(0.5,0.4) (-0.2,0.3) (-0.55,-0.25)};
\draw [->]  (0.45,0.4)--(0.5,0.4);
\draw [->]  (-0.5,-.15)--(-0.55,-0.25);
\foreach \Point in {(0,-2.5)}{ \node at \Point {\textbullet};}
\node (A) at (-0.2, 0.5) {$\tilde{\sigma}$};
\node (X) at (-1.8, -2.5) {$X$};
\node (C) at (-1.8, 0) {$\tilde{X}_s $};
\node (D) at (0, -.3) {$\lambda$};
\node (p) at (0, -2.8) {$p$};
\node (x) at (0, 1.5) {$x$};
\node (t) at (1.6, 0) {$t$};
\end{tikzpicture}  $$
If $r$ is odd, $\hat{\sigma}$ fixes also the fixed points of $\tilde{\sigma}$ outside the singular points.\\

Let $\varepsilon(r)=r\mod 2$. It follows that the genus $g_{\hat{Y}_s}$ of $\hat{Y}_s=\hat{X}_s/\hat{\sigma}$ is given by \begin{align*} g_{\hat{Y}_s}&=\frac{1}{2}\left(g_{\hat{X}_s}+1-\dfrac{r-2k_p}{2}-\varepsilon(r)\dfrac{2n-1}{2}\right)\\ &=\dfrac{r^2}{2}(g_X-1)+1-\dfrac{(r-2k_p)^2}{4}+\varepsilon(r)\dfrac{2n-1}{4}\\&= r^2(g_Y-1)+1+n\dfrac{r^2}{2}-\dfrac{(r-2k_p)^2}{4}+\varepsilon(r)\dfrac{2n-1}{4}\\&= r^2(g_Y-1)+1+(2n-1)\dfrac{r^2-\varepsilon(r)}{4}+k_p(r-k_p).
\end{align*}
This implies $$g_{\hat{Y}_s}=\text{dim(Pic}^m(\hat{X}_s)^{\hat{\sigma}})=\text{dim}(\al U_X^{\sigma,\tau}(r,0)).$$ So the generic fiber of the Hitchin morphism  $$\sr H:T^*\al U_X^{\sigma,\tau}(r,0)\lra W^{\sigma,\tau}$$ is of maximal dimension. So we get again the complete integrability of the Hitchin system in this case too.

Moreover, with the same method used so far proving the dominance results, we deduce 

\begin{theo}
For each type $\tau$, the pushforward map $$\text{Pic}^m(\hat{X}_s)^{\hat{\sigma},\hat{\tau}}\dashrightarrow \al U_X^{\sigma,\tau}(r,0)$$
is dominant, for some type  $\hat{\tau}$ of $\hat{\sigma}-$invariant line bundles over $\hat{X}_s$.
\end{theo}
\begin{proof}
First the type $\hat{\tau}$ is constructed as follows: If $r$ is even, $\hat{\sigma}$ has $r-2k_p$   fixed points which are all over $p$, then since $k_p$ is chosen strictly smaller than $[r/2]$, we take the trivial type. If $r$ is odd, then over any ramification point $a$ such that $k_a=\frac{r+1}{2}$ the type is equal $-1$, and over the rest of ramification points other than $p$ the type is $+1$. However over $p$ there are $r-2k_p$ fixed points by $\hat{\sigma}$, the types over these points are all equal $+1$.
  
Now we deduce as in the proof of Theorem \ref{main1} that the map $$\Pi:T^*\al U_X^{\sigma,\tau}(r,0)\lra\al U_X^{\sigma,\tau}(r,0)\times W^{\sigma,\tau}$$ is dominant. So for general $s\in W^{\sigma,\tau}$, the fiber $\sr H^{-1}(s)$ dominates $\al U_X^{\sigma,\tau}(r,0)$. Moreover $\sr H^{-1}(s)$ is identified, by Proposition \ref{bnrprop}, with a set of torsion-free rank one sheaves over $\tilde{X}_s$, which are $\tilde{\sigma}-$invariant. Let $x_0$ be the singular point of $\tilde{X}_s$ over $p$.  Then twisting these torsion-free sheaves with $$\al O_{\tilde{X}_s}(-\frac{(r-2k_p)(r-2k_p-1)}{2}x_0)$$  and pulling them back to $\hat{X}_s$, we identify $\sr H^{-1}(s)$ with the open subset of $\text{Pic}^m(\hat{X}_s)^{\hat{\sigma},\hat{\tau}}$ of line bundles such that $\hat{q}_*L$ is stable, where $m$ is the degree of $(\hat{q}_*\al O_{\hat{X}_s})^*$ ($m=r(r-1)(g_X-1)-\frac{(r-2k_p)(r-2k_p-1)}{2}$). The result follows.
\end{proof}

\section{Appendices}
\subsection{Appendix A}\label{appA}
In this appendix, we construct stable anti-invariant vector bundles. Let $\beta\in J_X[r]$ a  primitive $r-$torsion point of the Jacobian which descends to $Y$, so in particular we assume that the genus $g_Y$ of $Y$ is at least $1$. Denote by $q:X_\beta\lra X$ the associated cyclic \'etale cover of $X$ of degree $r$, and by $\iota$ a generator of the Galois group $\text{Gal}(X_\beta/X)$.
\begin{lemm}
The involution $\sigma:X\ra X$ lifts to an involution $\tilde{\sigma}:X_\beta\ra X_\beta$. Moreover, if $r$ is even, there are two such liftings of $\sigma$ such that one of them has no fixed points, we denote it by $\tilde{\sigma}_-$. 
\end{lemm}
\begin{proof}
The curve $X_\beta$ is a spectral curve given by the equation $x^r-1=0$ in the ruled surface $\bb P(\al O_X\oplus \beta^{-1})$. As in the proof of Proposition \ref{descentram}, the positive linearisation on $\beta$ gives an involution $\tilde{\sigma}$ on $X_\beta$ that lifts $\sigma$. If $r$ is even, then the negative linearisation gives also a lifting of $\sigma$. One remarks that $q(\text{Fix}(\tilde{\sigma}))\subset \text{Fix}(\sigma)$, hence if $\pi:X\ra Y$ is \'etale, then $X_\beta\ra X_\beta/\tilde{\sigma}$ is \'etale too. However, if $r$ is even, the negative linearisation has no fixed point because its only fixed point is $0$ and $0$ is not a root of $x^r-1=0$.
\end{proof}

\begin{prop} \label{stabledirectimage}
The line bundles of degree $0$ on $X_\beta$ such that $q_*L$ is not stable are those with a non-trivial stabiliser subgroup of $\gen{\iota}$.
\end{prop}
\begin{proof}
This is true for any Galois cover, it is proved in the (unpublished) paper of Beauville entitled "On the stability of the direct image of a generic vector bundle".\\ 
Let $L\in\text{Pic}^0(X_\beta)$ such that $q_*L$ is not stable. Let $F\hookrightarrow q_*L$ be a stable subbundle  of degree $0$, it follows $$q^*F\hookrightarrow q^*q_*L=L\oplus \iota^*L\oplus\cdots\oplus(\iota^{r-1})^*L,$$
hence $q^*F$ is of the form $\bigoplus_{j\in J}(\iota^j)^*L$ for some $J\subsetneqq \{0,\cdots,r-1\}$. In particular $q_*L$ is semistable. On the other hand, The adjunction formula gives a non-zero map $q^*F\ra (\iota^k)^*L$ for any $k$. As $q^*F$ is semistable of degree $0$, this map is surjective. Hence  $\bigoplus_{j\in J}(\iota^j)^*L\ra (\iota^k)^*L$ is surjective for any $k$, it follows that there exists $k\in\{1,\cdots,r-1\}$ such that $(\iota^k)^*L\cong L$. So $\iota ^k$ is in the stabiliser of $L$. \\
Conversely, let $L$ such that $(\iota^k)^*L\cong L$ for some $0<k<r$, it follows that $\iota^*L\oplus\cdots\oplus (\iota^k)^*L$ is $\iota-$invariant, so it descends to a vector bundle, say $F$, on $X$,  as $\text{deg}(F)=0$, by adjunction, we deduce that $F\hookrightarrow q_*L$, hence $q_*L$ is not stable.
\end{proof}
Now we can construct some stable anti-invariant vector bundles.
\begin{prop}
\begin{enumerate}
\item There exist stable $\sigma-$symmetric and $\sigma-$alternating anti-invariant vector bundles. 
\item The determinant maps $$\text{det}:\al U_X^{\sigma,+}(r)\ra P^+=\text{Nm}^{-1}(\al O_Y),$$ $$\text{det}:\al U_X^{\sigma,-}(r)\ra P^-=\begin{cases} \text{Nm}^{-1}(\al O_Y)& r\equiv0\mod2\\ \text{Nm}^{-1}(\Delta)& r\equiv1\mod2 \text{ and } \pi \text{ \'etale} \end{cases},$$ are surjective.
\end{enumerate}
\end{prop}
\begin{proof}
\begin{enumerate}
\item We denote $Y_\beta=X_\beta/\tilde{\sigma}$ and $Z_\beta=X_\beta/\tilde{\sigma}_-$ if $r$ is even. By Proposition \ref{stabledirectimage} we deduce that a general element in $\text{Nm}_{X_\beta/Y_\beta}^{-1}(\al O_{Y_\beta})$ has a stable direct image which is $\sigma-$symmetric. Let $\Delta_\beta$ is the $2-$torsion point attached to $X_\beta\ra Z_\beta$, then a general element in $\text{Nm}_{X_\beta/Z_\beta}^{-1}(\Delta_\beta)$ has a stable direct image which is $\sigma-$alternating, and if $r$ is odd and $\pi:X\ra Y$ is \'etale, a general element in $\text{Nm}_{X_\beta/Y_\beta}^{-1}(\Delta_\beta)$ has a stable direct image which is also $\sigma-$alternating.\\
\item If $\pi$ is ramified, or $\pi$ is \'etale and $r$ is odd, then this is clear due to taking the tensor product of a fixed anti-invariant vector bundle by elements of $P^\pm$. Assume that $\pi$ is \'etale and $r$ is even. By definition $P^+=P^-$, so we denoted just by $P$. Taking the tensor product by elements of $P$ does not make the determinant surjective, so we need to prove the existence of stable vector bundles whose determinants are in both connected components of $P$. But one remarks that $\text{Nm}_{X_\beta/X}:\al P^\pm\ra P$ is surjective, where $\al P^+=\text{Nm}_{X_\beta/Y_\beta}^{-1}(\al O_{Y_\beta})$  and $\al P^-=\text{Nm}_{X_\beta/Y_\beta}^{-1}(\Delta_\beta)$, as we have $\text{det}(q_*L)=\text{Nm}_{X_\beta/X}(L)\otimes \beta^{r(r-1)/2}$, and because $\beta^{r(r-1)/2}\in P$, this ends the proof. 
\end{enumerate}
\end{proof}

\subsection{Appendix B}
In this appendix, we give another proof of the irreducibility of $\al U_X^{\sigma,+}(2)$ and the fact that $\al U_X^{\sigma,-}(2)$ has two connected components in the ramified case. \\ Assume that $\pi:X\ra Y$ is ramified. In this case every vector bundle $E$ over $ X$ with trivial determinant is a $\text{Sp}_2-$bundle, that's $E\cong E^*$ with a symplectic form, and it admits a symmetric one if and only if it is polystable.\\
We see in this particular case that $\sigma-$invariant bundles are the same as $\sigma-$anti-invariant bundles, so let $(E,\phi)$ be a stable $\sigma-$invariant bundle with trivial determinant. The triviality of the determinant implies that the type  of $E$ must be of the form $$\tau=(\mathbf{0},\cdots,\mathbf{0}),$$ $$\text{or  } \tau=(A_1,\cdots,A_{2n}), \text{with } A_i\in\{\mathbf{+1},\mathbf{-1}\},$$
where $\mathbf{0}=\begin{pmatrix}
-1 &0 \\0&1
\end{pmatrix}$ and $\pm\mathbf{1}=\begin{pmatrix}
\pm1& 0\\ 0 & \pm1
\end{pmatrix}$. \\
We have $$\psi:\sigma^* E\stackrel{\phi}{\longrightarrow} E\stackrel{q}{\longrightarrow} E^*,$$ 
where $q$ is a symplectic form, let $\psi=q\circ \phi$. It is not difficult to see that if $E$ has a type $(\mathbf{0},\cdots,\mathbf{0})$, then $\psi$ is $\sigma-$symmetric, and it is $\sigma-$alternating otherwise. \\
In particular, one deduces that $\al{SU}_X^{\sigma,+}(2)$ is connected, and $\al {SU}_X^{\sigma,-}(2)$ has $2^{2n-1}$ connected components.
 
 Moreover, we have surjective maps $$\al{SU}_X^{\sigma,+}(2)\times P\ra \al U_X^{\sigma,+}(2),$$ $$\al{SU}_X^{\sigma,-}(2)\times P\ra \al U_X^{\sigma,-}(2).$$
This proves the irreducibility of $\al U_X^{\sigma,+}(2)$.\\
The group $P[2]$ acts naturally on the set of connected components of $\al{SU}_X^{\sigma,-}(r)$ in the following way: for $\lambda\in P[2]$ of type $\upsilon=(\varepsilon_1,\cdots,\varepsilon_{2n})$, where $\varepsilon_i\in\{\pm1\}$, then for a type $\tau=(A_1,\cdots,A_{2n})$ attached to a connected component of $\al{SU}_X^{\sigma,-}(r)$, we have $$\upsilon\cdot\tau=(\varepsilon_1A_1,\cdots,\varepsilon_{2n}A_{2n}).$$   Furthermore this action is free modulo elements of $\pi^*J_Y[2]$. Since $\text{card}(P[2]/\pi^*J_Y[2])=2^{2n-2}$, we deduce that this action has two orbits. It follows in particular that $\al U_X^{\sigma,-}(2)$ has two connected components.

\bibliographystyle{alpha}
\bibliography{bib}
\end{document}